\documentclass{article}

 \usepackage[preprint]{neurips_2026}

\usepackage[utf8]{inputenc} 
\usepackage[T1]{fontenc}    
\usepackage{hyperref}       
\usepackage{url}            
\usepackage{booktabs}       
\usepackage{amsfonts}       
\usepackage{nicefrac}       
\usepackage{microtype}      
\usepackage{xcolor}         

\usepackage[normalem]{ulem}
\usepackage{booktabs, array, makecell}
\usepackage{graphicx} 
\usepackage{algorithm}
\usepackage{algpseudocode}
\usepackage{amsthm}
\usepackage{lipsum}
\usepackage{epstopdf}
\usepackage{mathtools}
\usepackage{enumitem}
\usepackage{soul}
\usepackage{multirow}
\usepackage[capitalize,noabbrev]{cleveref}
\usepackage{amssymb}
\usepackage{tikz}
\newcommand{\R}{\mathbb{R}}

\newcommand{\N}{\mathbb{N}}
\newtheorem{assumption}{Assumption}

\newtheorem{lemma}{Lemma}
\newtheorem{theorem}{Theorem}
\newtheorem{definition}{Definition}
\newtheorem{remark}{Remark}

\newcommand{\KL}{\mathcal{KL}}
\newcommand{\KK}{\mathcal{K}}
\newcommand{\KKinf}{\mathcal{K}_\infty}

\title{From Cursed to Competitive: Closing the ZO–FO Gap via Input-to-State Stability}

\author{%
  Amir Ali Farzin
    \\
  School of Engineering\\
  Australian National University\\
  Canberra, ACT 2601 \\
  \texttt{amirali.farzin@anu.edu.au} \\
  \And
  Philipp Braun \\
  School of Engineering \\
  Australian National University\\
  Canberra, ACT 2601 \\
  \texttt{philipp.braun@anu.edu.au} \\
  \AND
  Iman Shames \\
  Department of Electrical and Electronic Engineering\\
  University of Melbourne \\
  Melbourne, VIC 3010 \\
  \texttt{iman.shames@unimelb.edu.au} \\
}

\begin{document}

\maketitle

\begin{abstract}
While it is generally understood that zeroth-order (ZO) algorithms have an extra dependency on their number of iterations for any choice of parameters, compared to their first-order (FO) counterparts, in this work, we show that under several conditions, in expectation, ZO methods do not suffer from extra dimension dependencies in their convergence rates with respect to their FO counterparts. We look at optimisation algorithms from the dynamical systems perspective and analyse the conditions under which one can formulate the average of a ZO algorithm as the average of its FO counterpart with bounded perturbations with values dependent on design parameters. Then, using input-to-state stability properties, we show ZO methods follow the same decay rate as their FO counterparts and converge to a neighbourhood of the fixed point of FO methods, where its radius depends on the bound of the norm of the perturbations, which can be made arbitrarily small. The theoretical findings are illustrated via numerical examples.
\end{abstract}

\section{Introduction}\label{sec:intro}

In this paper, we study a minimisation problem of the form
\begin{equation}\label{eq:main}
	\min_{x\in \mathbb{R}^n} \quad f(x)
\end{equation}
where $f:\mathbb{R}^n\rightarrow \mathbb{R}$ is continuously differentiable, and bounded below. Such problems arise routinely in machine learning, signal processing, and control. While previously, in the literature, it was commonly understood
that convergence complexity bounds for zeroth-order (ZO) methods have extra dimension dependencies compared to their first-order (FO)
counterparts, in this work, we analyse 
conditions under which it is guaranteed that ZO methods, in expectation, have no extra dimension dependencies in their convergence rate compared to their FO counterparts.
Here, we study iterative algorithms and their convergence properties
from the dynamical system perspective and analyse them by studying input-to-state stability (ISS) properties.

\textit{Motivation:} In many modern applications, gradient information is either unavailable, prohibitively expensive to compute, or simply does not exist in closed form. Examples include black-box adversarial attacks~\citep{chen2017zoo,ye2018hessian,madry2017towards}, reinforcement learning with opaque simulators~\citep{salimans2017evolution,choromanski2018structured}, hyperparameter tuning~\citep{snoek2012practical}, simulation-based engineering design~\citep{conn2009introduction}, and memory-efficient fine-tuning of large language models~\citep{malladi2023fine}. In these settings, one must resort to ZO methods that construct gradient surrogates from function evaluations alone. While attractive for their simplicity and broad applicability, ZO methods have traditionally been seen as having fundamentally weaker convergence guarantees than their FO counterparts, independent of the parameter selection.

\textit{The dimension-dependence gap:} The seminal work of~\cite{nesterov_random_2017} established the modern framework for ZO optimisation based on Gaussian smoothing and showed that, for smooth convex problems, ZO gradient descent requires at most $n$ times more iterations than standard gradient descent, where $n$ is the dimension. This $\mathcal{O}(n)$ factor has become the standard penalty attributed to working without gradient information. Subsequent works have refined and extended these bounds across convex~\citep{duchi2015optimal,balasubramanian2018zeroth}, nonconvex stochastic~\citep{ghadimi2013stochastic}, and structured non-convex settings such as quasar-convex, submodular, and Polyak-\L{}ojasiewicz functions~\cite{farzin2025minimisationqc,farzin2025minimisation,farzin2024minimisation}. In each case, the convergence bounds carry an explicit dependence on $n$ beyond what FO methods require. Efforts to reduce this overhead via sparsity~\citep{wang2018stochastic,balasubramanian2022zeroth} or effective dimension~\citep{yue2023zeroth} typically require additional structural assumptions (e.g., gradient sparsity or a fast-decaying Hessian spectrum) that may not hold in general.

\textit{Related work:} On the ZO optimisation side, building on~\cite{nesterov_random_2017,ghadimi2013stochastic}, ZO variants of classical FO algorithms have been developed, including accelerated methods~\citep{dvurechensky2020accelerated}, variance-reduced schemes~\citep{liu2018zeroth}, and minimax methods~\citep{liu2020min,wang2020zeroth,farzin2025min,farzin2026solving}. Notably, the MeZO framework~\citep{malladi2023fine} demonstrated that ZO-SGD can fine-tune LLMs with billions of parameters at competitive performance, attributed to the favourable loss landscape induced by pre-training; subsequent work~\citep{zhang2024revisiting} benchmarked broader ZO optimiser families in this setting. On the dynamical systems side, interpreting optimisation algorithms as discrete-time dynamical systems is well established~\citep{lessard2016analysis,wilson2021lyapunov,farzin2026stability,daskalakis2018limit,farzin2025properties}. Input-to-state stability (ISS), introduced in~\cite{sontag1989smooth} and extended to discrete-time systems in~\cite{jiang2001input}, provides a natural framework for analysing robustness of stable equilibria to bounded perturbations and has been widely applied in nonlinear control~\citep{sontag2008input,kellett2023introduction}. Its application to bridge ZO and FO convergence theory is, to the best of our knowledge, novel. A detailed comparison of our results with those of prior ZO methods is provided in Table~\ref{tab:comparison} (Appendix~\ref{app:def}).

\textit{Contributions:} The main contributions of this paper are as follows. \\
    \hspace*{0.2cm}(i) We show that, under suitable conditions, the averaged dynamics of a Gaussian ZO algorithm approximate its FO counterpart, where the approximation error is bounded and can be made arbitrarily small by tuning the smoothing parameter $\mu$, the step size $h$, and the number of sampled directions $t$.\\
    \hspace*{0.2cm}(ii) Leveraging the ISS framework, we prove that ZO methods inherit the same transient convergence rate as their FO counterparts, i.e., the same exponential/asymptotic decay, and converge to a neighbourhood of the FO fixed point whose radius depends on the perturbation bound but carries no additional dependence on the number of iterations (Theorems~\ref{thm:ISS_exp} and~\ref{thm:ISS_asymp}). This provides a fundamentally different explanation of ZO convergence behaviour. The dimension dependence does not appear in the iteration complexity but rather in the size of the convergence neighbourhood. Moreover, the framework serves as a general-purpose tool. Whenever the FO algorithm is known to converge, the ISS machinery automatically yields convergence of its ZO counterpart to a neighbourhood of the FO fixed point, without requiring a separate, algorithm-specific proof. Our results concern \emph{iteration complexity}. The per-iteration cost of each ZO step involves $t$ function evaluations. It may depend on problem parameters, including the variance, to achieve a given neighbourhood radius.\\
    \hspace*{0.2cm}(iii) We apply this framework to GD, the heavy ball, and Nesterov's accelerated gradient. Under strong convexity, we derive explicit perturbation bounds (Lemmas~\ref{lem:q_GDsc},~\ref{lem:q_HBsc},~\ref{lem:q_NAGsc}) and dimension-free convergence rate results (Theorems~\ref{th:GDsc},~\ref{th:HBsc},~\ref{th:NAGsc}). For non-strongly-convex objectives, we show that $L_2$ regularisation provides sufficient contraction, enabling the same dimension-free analysis with a controlled trade-off between regularisation strength and neighbourhood size (Lemma~\ref{lem:q_GDreg} and Theorem~\ref{th:GD_reg}). \\ 
    \hspace*{0.2cm}(iv) We validate the theoretical findings through numerical experiments, confirming that ZO trajectories closely track their FO counterparts and demonstrating how $\mu$, $h$, and $t$ control the proximity.

\textit{Outline:} 
Preliminaries
are introduced in Section~\ref{sec:per}. In Section~\ref{sec:main},  the main results on the dimensional independency of ZO methods are given. In Section~\ref{sec:exps}, some illustrative examples are studied. The conclusions and future research directions are discussed in Section~\ref{sec:conc}. Additional details and the proof of the technical results can be found in Appendices~\ref{app:def}-~\ref{app:exps}.

\textit{Notation:}
For $x,y\in\R^n$, we write $\langle x,y\rangle = x^\top y$, and let $\|\cdot\|$ denote the Euclidean norm of a vector and the corresponding spectral norm of a matrix. For $\delta>0$ and $z^\ast \in \R^n$, we define $\mathcal{B}_\delta(z^\ast) := \{z\in\R^n \mid \|z-z^\ast\|\le \delta\}$. For $f:\R^{n}\to\R$, we denote the gradient and Hessian by $\nabla f$ and $\nabla^2 f$, respectively. The identity matrix of proper size is denoted by $\mathrm{I}$, and $\mathrm{Id}$ denotes the identity function $\mathrm{Id}(s)=s$. A continuous function $\alpha:[0,a)\to[0,\infty)$ (with $a>0$ or $a=\infty$) is of class $\mathcal{K}$ if it is strictly increasing and $\alpha(0)=0$. A function $\alpha$ is of class $\mathcal{K}_\infty$ if $\alpha\in\mathcal{K}$ and $\alpha(s)\to\infty$ as $s\to\infty$. A continuous function $\beta:[0,\infty)\times[0,\infty)\to[0,\infty)$ is of class $\mathcal{KL}$ if, for each fixed $t\ge 0$, the map $s\mapsto\beta(s,t)$ belongs to $\mathcal{K}$, and for each fixed $s\ge 0$, the map $t\mapsto\beta(s,t)$ is decreasing with $\beta(s,t)\to 0$ as $t\to\infty$.

\section{Preliminaries}\label{sec:per}
In this section, we present
preliminaries on dynamical systems, ISS,
and ZO Gaussian 
oracles.

\subsection{Dynamical systems and perturbations}
We study optimisation problem~\eqref{eq:main} through the lens of discrete-time dynamical systems, interpreting an algorithm as
\begin{align}
    z_{k+1} = w(z_k) \qquad (\text{or } z^+ = w(z)) ,
    \label{eq:discrete_time_sys}
\end{align}
with state $z\in\R^n$ and update map $w:\R^n\to\R^n$. A point $z^e\in\R^n$ is an equilibrium (fixed point) of $w$ if $w(z^e)=z^e$. To analyse stability of equilibria, we use standard notions from discrete-time systems theory; see, e.g.,~\cite[Ch.~2, Ch.~5]{kellett2023introduction}.

\begin{definition}\label{def:stability}
    An equilibrium $z^e\in\R^n$ of~\eqref{eq:discrete_time_sys} is \emph{stable} if there exist $\alpha\in \mathcal{K}_{\infty}$ and $\delta>0$ such that
    $\|z_k - z_e\| \leq \alpha(\|z_0 - z_e\|)$ for all $k\in \N$ and $z_0 \in \mathcal{B}_{\delta}(z^e)$.
    If $z^e$ is not stable, it is called \emph{unstable}.
    The equilibrium is \emph{asymptotically stable} 
    if there exist $\beta\in \mathcal{KL}$ and $\delta>0$ such that
    \begin{align}
        \|z_k - z_e\| \leq \beta(\|z_0 - z_e\|,k) \qquad \forall \ k\in \N, \quad \forall \ z_0 \in \mathcal{B}_{\delta}(z^e). \label{eq:asymp_stab} 
    \end{align}
    If $\beta$ in \eqref{eq:asymp_stab} has the form $\beta(r,k)=M r \nu ^k$ for $M>0$ and $\nu \in (0,1)$, then $z^e$ is \emph{exponentially stable}.
\end{definition}

For an optimisation algorithm, $\mathcal{KL}$-stability or asymptotic stability of an equilibrium of~\eqref{eq:discrete_time_sys} is equivalent to local convergence to the associated critical point. We now consider a perturbed version of~\eqref{eq:discrete_time_sys}:
\begin{align}
    z_{k+1} = \bar{w}(z_k,q_k) = w(z_k) + q_k ,
    \label{eq:dt_perturbed}
\end{align}
where $\{q_k\}_{k\in\N}$ models perturbations (e.g., inexact gradients or noise). When these perturbations are bounded, robustness can be studied via ISS~\cite{sontag1996new,jiang2001input}.
\begin{definition}[Input-to-state stability] 
The system~\eqref{eq:dt_perturbed} with $z_k,q_k\in\R^n$ is (locally) input-to-state stable (ISS) if there exist $\delta>0$, $\beta\in\mathcal{KL}$, and $\gamma\in\mathcal{K}$ such that, for all sequences $\{q_k\}_{k\ge 0}$,
\begin{equation}
    \|z_k - z^e\| \le \beta(\|z_0 - z^e\|,k) + \gamma\Big(\sup_{0\le j\le k} \|q_j\|\Big) \qquad  \forall \ k\in \N, \quad \forall \ z_0\in\mathcal{B}_\delta(z^e).
    \label{eq:iss_estimate}
\end{equation}
\end{definition}
The ISS estimate~\eqref{eq:iss_estimate} ensures that the state remains bounded for bounded inputs and recovers asymptotic stability when $q_k\equiv 0$~\cite{sontag1996new,kellett2023introduction}. We adopt the additive perturbation structure in~\eqref{eq:dt_perturbed}, which is sufficient for our analysis and simplifies the ISS arguments. A convenient way to verify ISS is via Lyapunov functions. The following theorems show that~\eqref{eq:dt_perturbed} inherits the transient decay of~\eqref{eq:discrete_time_sys} and converges to a neighbourhood of its equilibria, whenever the perturbations are uniformly bounded and the unperturbed system is exponentially or asymptotically stable.

\begin{theorem}\label{thm:ISS_exp}
Consider the  
system \eqref{eq:dt_perturbed}
with equilibrium $z^e \in \R^n$ for $q_k=0$ for $k\in \N$, and suppose there exist a  
locally Lipschitz function $V:\R^n\rightarrow \R_{\geq 0}$,
constants $c\in(0,1)$ and $R>0$, and constants $c_1,c_2\in \R_{>0}$,
such that
\begin{alignat}{3}
c_1 \|z-z^e\|^2 \le V(z) &\le c_2\|z-z^e\|^2,  && \qquad \forall \ z \in \mathcal{B}_R(z^e) \label{eq:sandwich}\\
    V(\bar{w}(z,0)) - V(z) &\le -c\,V(z), && \qquad \forall \ z \in \mathcal{B}_R(z^e) \label{eq:decrease}
\end{alignat}
whenever $q_k=0$ for all $k\in \N$.
Then there exist a class-$\KK$ function $\gamma$ and radii $\bar{r},\bar{q}>0$ such that, for the perturbed system \eqref{eq:dt_perturbed}
with $\sup_k\|q_k\|\le \bar{q}$ and initial condition satisfying $z_0\in \mathcal{B}_{\bar{r}}(z^e)$, 
the corresponding solution satisfies
\begin{equation}\label{eq:ISS_bound}
 \|z_k-z^e\|
\le \sqrt{\frac{c_2}{c_1}}\|z_0-z^e\|(1-c)^{\frac{k}{2}}
   +\gamma\Bigl(\sup_{0\le j\le k}\|q_j\|\Bigr),
\qquad k\ge 0.
\end{equation}
In particular, $z^e$ is locally exponentially stable for $q_k = 0$, $k\in \N$, and locally ISS with respect to 
perturbations $q_k\in \mathcal{B}_{\bar{q}}(0)$, $k\in \N$. 
\end{theorem} 

The function $V$ in Theorem \ref{thm:ISS_exp} is called Lyapunov function.
A proof of Theorem~\ref{thm:ISS_exp} is given in Appendix~\ref{app:per}. The above theorem shows that the perturbed system exhibits the same local exponential decay rate as the unperturbed system and converges to a neighbourhood of the equilibrium of the unperturbed system, where the size of the neighbourhood depends on the supremum of the norm of the perturbations. For extension to the setting where the unperturbed system is asymptotically stable but not exponentially,  
we first present the following lemma.

\begin{lemma}\label{lem:comparison}
Let $\varphi\colon[0,\infty)\to[0,\infty)$ be continuous, strictly increasing, with $\varphi(0)=0$, and $\mathrm{Id}-\varphi$ a class-$\KK$ function. Let $\{v_k\}_{k\ge 0}$ be a non-negative sequence satisfying
\[
v_{k+1}\le \varphi(v_k)+d, \qquad \forall \  k\ge 0,
\]
for some constant $d \in [0,b],$ where $b = \underset{r\rightarrow\infty}{\lim} r-\varphi(r)$. Then the following properties are satisfied. \\

\hspace*{0.2cm} (i) There exists a class-$\KL$ function $\hat\beta$ such that, if $d=0$, then $v_k\le\hat\beta(v_0,k)$ for all $k\ge 0$. \\
\hspace*{0.2cm} (ii) For every $\varepsilon>0$, there exists $\bar d(\varepsilon)>0$ such that, if $d\le\bar d(\varepsilon)$, then $\limsup_{k\to\infty}v_k\le\varepsilon$. \\
\hspace*{0.2cm} (iii) For every $R>0$, there exists a class-$\KK$ function $\hat\gamma_R$ such that, with $\hat\beta$ as in~\textup{(i)},
\[
v_k \le \hat\beta(v_0,k)+\hat\gamma_R(d), \qquad \forall \  k\ge 0, \quad \forall \  v_0\in[0,R], \quad \forall \ d\ge 0.
\]

\end{lemma}


\begin{theorem}\label{thm:ISS_asymp}
Consider the 
system \eqref{eq:dt_perturbed}
with equilibrium $z^e \in \R^n$ for $q_k=0$ for $k\in \N$, and suppose there exist a locally Lipschitz
function $V:\R^n\rightarrow \R_{\geq 0}$, 
class-$\KKinf$ functions $\alpha_1,\alpha_2,\alpha_3$, and a constant $R>0$
such that
\begin{alignat}{3}
    \alpha_1(\|z-z^e\|) \le V(z) & \le \alpha_2(\|z-z^e\|), \qquad  && \forall \ z \in \mathcal{B}_R(z^e), \label{eq:sandwich1}\\
    V(\bar{w}(z,0)) - V(z) &\le -\alpha_3(\|z-z^e\|), \qquad  && \forall \ z \in \mathcal{B}_R(z^e). \label{eq:decrease1}
\end{alignat}
Then there exist a class-$\KL$ function $\beta$, a class-$\KK$ function $\gamma$, and radii $\bar{r},\bar{q}>0$ such that, for the perturbed
system \eqref{eq:dt_perturbed}
with $\sup_k\|q_k\|\le \bar{q}$ and $z_0\in \mathcal{B}_{\bar{r}}(z^e)$, 
the solution satisfies
\begin{equation}\label{eq:ISS}
\|z_k - z^e\| \le \beta(\|z_0-z^e\|,\,k) + \gamma \Big(\sup_{0\le j\le k}\|q_j\|\Big), \qquad k\ge 0.
\end{equation}
\end{theorem}

Proofs of Lemma~\ref{lem:comparison} and Theorem~\ref{thm:ISS_asymp} are given in Appendix~\ref{app:per}. Together with Theorem~\ref{thm:ISS_exp}, they show that the perturbed system converges to a neighbourhood of the equilibrium of the unperturbed system with the same transient decay profile, without any extra dependence on the iteration index $k$, where the neighbourhood size is determined by the perturbation bound. This provides a general mechanism for transferring convergence guarantees from FO algorithms to their ZO counterparts. Whenever an FO algorithm is known to converge, one only needs to verify that the averaged ZO dynamics can be written in the form~\eqref{eq:dt_perturbed} with bounded perturbations, and the ISS machinery automatically guarantees convergence to a neighbourhood of the same fixed point, inheriting the same decay profile. In Section~\ref{sec:main}, we analyse conditions under which such a formulation is possible. Crucially, since the perturbations originate from the ZO approximation rather than a model mismatch, disturbances or sensor noise, their magnitude can be manipulated through appropriate parameter selections.

\subsection{Gaussian smoothing}
Following \cite{nesterov_random_2017}, we define the Gaussian-smoothed version of $f$ with parameter $\mu>0$ as $f_\mu(x) = E_u[f(x+\mu u)]$, where $u\sim\mathcal{N}(0,I)$. Regardless of whether $f$ is differentiable, $f_\mu$ is always differentiable with $\nabla f_\mu(x) = E_u\!\left[\frac{f(x+\mu u)}{\mu}u\right]$ \cite{nesterov_random_2017}. We use the two-point random oracle
\begin{equation}\label{eq:eq23}
	g_{\mu}(x) = \tfrac{1}{\mu}(f(x+\mu u)-f(x))\, u,
\end{equation}
which satisfies $E_{u}[g_{\mu}(x)]=\nabla f_\mu(x)$. We call an algorithm a Gaussian ZO method if it uses~\eqref{eq:eq23} in place of gradients.
\begin{assumption}\label{lm20}
	The variance of the oracle $g_{\mu}$ defined in~\eqref{eq:eq23} is uniformly upper bounded by $\sigma^2\geq0$; i.e., $E_{u}[\|g_{\mu}(x)- \nabla f_{\mu}(x)\|^2]\leq \sigma^2$.
\end{assumption}
This is a standard assumption; see, e.g., \cite{maass2021zeroth,liu2020min,farzin2025min}. The variance can be reduced by averaging $t$ independent oracle evaluations per iteration, yielding $E_{u}\!\left[\|g_{\mu}(x_k)- \nabla f_{\mu}(x_k)\|^2\right]\leq \sigma^2/t$ while preserving unbiasedness~\citep{balasubramanian2022zeroth}. 
We note that since our convergence results are local (Theorems~1 and~2 hold on a ball $B_{\bar{r}}(z^e)$), Assumption~1 needs only to hold on this compact region. A relaxation to state-dependent variance is analysed in Remark~\ref{rem:relaxed_variance}.

\section{ZO methods as perturbed FO methods}\label{sec:main}

Without further assumptions, there is no guarantee that the expectation of a Gaussian ZO method can be formulated as a perturbed FO method with characterisable bounded 
perturbations. Here, we identify conditions under which such a formulation is possible. Then, Theorems~\ref{thm:ISS_exp} and~\ref{thm:ISS_asymp} imply that the convergence of the Gaussian ZO method has no extra dependence on the problem dimension and follows the same transient decay as the FO method. The proofs of all the lemmas and theorems of this section are given in Appendix~\ref{app:main}.

\subsection{Gradient descent and strong convexity}\label{sec:GDsc}
We focus on the
GD algorithm and assume that $f$ is strongly convex. The GD update rule is
\begin{align}\label{eq:GD_alg}
    z_{k+1} = z_k -h \nabla f(z_k),
\end{align}
where $h>0$ is the step size. Suppose $f$ has Lipschitz gradients with constant $L_1>0$. The ZO counterpart of GD is
\begin{align}\label{eq:ZOGD_alg}
    \tilde{z}_{k+1} = \tilde{z}_k -h g_\mu(\tilde{z}_k).
\end{align}
From \cite{nesterov_random_2017}, we know that $E[g_\mu]=\nabla f_\mu$. Let $U_k = \{u_1, \dots, u_k\}$ and define the averaged (deterministic) iterate of ZO-GD and its update rule as
\begin{align}\label{eq:ZOGD_avg1}
    \bar{z}_{k} = E_{U_k}[\tilde{z}_{k}], \qquad\bar{z}_{k+1} = \bar{z}_k -h E_{U_k}[\nabla f_\mu(\tilde{z}_k)].
\end{align}
This can be reformulated as
\begin{align}\label{eq:ZOGD_avg}
\bar{z}_{k+1} = \bar{z}_k -h \nabla f(\bar{z}_k) + q_k,
\end{align}
where
\begin{align}
\begin{split}
    q_k \! = \! h \bigl(\nabla f(\bar{z}_k) \! - \! E_{U_k}[\nabla f_\mu(\tilde{z}_k)]\bigr) \! = \! h \bigl(\nabla f(\bar{z}_k) \! - \!  \nabla f_\mu(\bar{z}_k)\bigr) \! + \! h\bigl(\nabla f_\mu(\bar{z}_k) \! - \! E_{U_k}[\nabla f_\mu(\tilde{z}_k)]\bigr).
\end{split} \label{eq:gen_q_ZOGD}
\end{align}

\begin{lemma}\label{lem:q_GDsc}
    Consider~\eqref{eq:GD_alg},~\eqref{eq:ZOGD_avg}, and~\eqref{eq:gen_q_ZOGD}. Let $f$ be $\beta$-strongly convex with Lipschitz gradients with constant $L_1>0$, and let $h\in\bigl(0,\frac{\beta}{2L_1^2}\bigr)$. Then
    \begin{align}\label{eq:q_bound_GDsc}
\sup_k \|q_k\| \leq h \frac{\mu}{2}L_1(n+3)^{3/2}+h L_1 \sqrt{\frac{2h\sigma^2}{t(\beta-2L_1^2h)}}.
\end{align}
\end{lemma}

Lemma~\ref{lem:q_GDsc} shows that the averaged ZO-GD iterates satisfy a perturbed version of the GD dynamics, where the perturbations are bounded and their magnitude can be manipulated through appropriate parameter selection. It is known that, for a smooth strongly convex objective, GD converges to the minimiser exponentially fast~\cite{nesterov2018lectures}. In~\cite{nesterov_random_2017}, ZO-GD was analysed under the same setting, and it was shown that the iteration complexity carries an extra dimension-dependent factor for any choice of parameters. The following theorem shows that, instead, the dimension affects only the size of the convergence neighbourhood, and there are parameter choices such that ZO-GD closely follows GD.

\begin{theorem}\label{th:GDsc}
Consider~\eqref{eq:GD_alg} and~\eqref{eq:ZOGD_avg}. Let $f$ be $\beta$-strongly convex with Lipschitz gradients with constant $L_1>0$, 
let $h\in\bigl(0,\frac{\beta}{2L_1^2}\bigr)$ 
and for $r>0$ let $\mathcal{B}_r(z^e)$ denote a convergence neighbourhood around the fixed point $z^e$ of GD.
Then there exist constants $\bar{\mu}, \bar{q} > 0$ and $\bar{t}\in\mathbb{N}$ (depending on $h$, $\beta$, $L_1$, and the Lyapunov function of the GD dynamics) such that, for any $\mu \in (0, \bar{\mu})$ and $t \geq \bar{t}$ satisfying $\sup_k \|q_k\| \leq \bar{q}$ via the bound in Lemma~\ref{lem:q_GDsc}, ZO-GD has the same exponential decay rate as GD  
converges to $\mathcal{B}_r(z^e)$, and 
there exists a class-$\mathcal{K}$ function $\alpha$ such that
\begin{align}\label{eq:r_GDsc}
r = \alpha\left( h \frac{\mu}{2}L_1(n+3)^{3/2}+h L_1 \sqrt{\frac{2h\sigma^2}{t(\beta-2L_1^2h)}} \right).
\end{align}
\end{theorem}

\subsection{Heavy ball method and Nesterov's accelerated gradient under strong convexity}\label{sec:HBNAGsc}

In this section, we consider both the heavy ball (HB) method and Nesterov's accelerated gradient (NAG) under strong convexity. The two algorithms share the same augmented state-space structure and differ only in the point at which the gradient is evaluated, so we treat them in a unified framework.

The HB and NAG update rules are, respectively,
\begin{align}
x_{k+1} &= x_k + h_2(x_k - x_{k-1}) - h_1 \nabla f(x_k) , \label{eq:heavyball-update}\\
x_{k+1} &= x_k + h_2(x_k - x_{k-1}) - h_1 \nabla f\bigl(x_k + h_2(x_k - x_{k-1})\bigr), \label{eq:nag-update}
\end{align}
and their ZO counterparts (ZO-HB, ZO-NAG) replace $\nabla f$ with $g_\mu$ defined in~\eqref{eq:eq23}.
Define the augmented state vector $z_k = \begin{bmatrix} x_{k-1}^\top & x_k^\top \end{bmatrix}^\top \in \mathbb{R}^{2n}$, the state transition matrix
\[
A = \begin{bmatrix} 0 & I_n \\ -h_2 I_n & (1+h_2) I_n \end{bmatrix},
\]
and the output map $T = [0 \;\; I_n] \in \mathbb{R}^{n \times 2n}$, so that $x_k = Tz_k$.
Note that $TAz_k = x_k + h_2(x_k - x_{k-1})$ is the lookahead point used by NAG.
Both algorithms can be written compactly as
\begin{equation}\label{eq:unified_update}
z_{k+1} = A z_k + \begin{bmatrix} 0 \\ -h_1 \nabla f(P z_k) \end{bmatrix},
\end{equation}
where $P = T$ for HB and $P = TA$ for NAG.
The ZO counterpart replaces $\nabla f(Pz_k)$ with $g_\mu(Pz_k)$.

Following the same steps as in Section~\ref{sec:GDsc}, conditioning on $z_k$, taking expectations, and adding and subtracting the gradient evaluated at the averaged state, we obtain the averaged dynamics
\begin{equation}\label{eq:avg_HBNAG_perturbed}
\bar{z}_{k+1} = A \bar{z}_k + \begin{bmatrix} 0 \\ -h_1 \nabla f(P\bar{z}_k) \end{bmatrix} + q_k,
\end{equation}
where the perturbation term is
\begin{equation}\label{eq:gen_q_unified}
q_k = \begin{bmatrix} 0 \\ -h_1 \Bigl( E_{U_k}[\nabla f_\mu(Pz_k)] - E_{U_k}[\nabla f(Pz_k)] + E_{U_k}[\nabla f(Pz_k)] - \nabla f(P\bar{z}_k) \Bigr) \end{bmatrix}.
\end{equation}
Equation~\eqref{eq:avg_HBNAG_perturbed} shows that the averaged trajectory follows the FO dynamical system (HB or NAG, depending on the choice of $P$) perturbed by $q_k$.

\begin{lemma}\label{lem:q_HBsc}
Consider~\eqref{eq:unified_update}--\eqref{eq:gen_q_unified} with $P=T$ (HB). Let $f$ be $\beta$-strongly convex with $L_1$-Lipschitz gradients, and let $h_1, h_2 > 0$ satisfy
\begin{equation}\label{eq:rho_HB_cond}
\rho_{\mathrm{HB}} := (1+2h_2)^2 - h_1(1+h_2)\beta + 4h_1 h_2 L_1 + 2h_1^2 L_1^2 < 1.
\end{equation}
Then
\begin{equation}\label{eq:q_bound_HBsc}
\sup_{k} \|q_k\| \leq \frac{h_1 \mu L_1(n+3)^{3/2}}{2} + \frac{h_1^2 L_1 \sigma}{\sqrt{t}} \sqrt{\frac{2}{1-\rho_{\mathrm{HB}}}}\,.
\end{equation}
A sufficient condition for~\eqref{eq:rho_HB_cond} is $h_2 \in (0, ch_1)$ and $h_1 \in \bigl(0, \frac{\beta - 4c}{4c^2 + 4cL_1 + 2L_1^2}\bigr)$ with $c \in \bigl(0, \frac{\beta}{4}\bigr)$.
\end{lemma}

\begin{lemma}\label{lem:q_NAGsc}
Consider~\eqref{eq:unified_update}--\eqref{eq:gen_q_unified} with $P=TA$ (NAG). Let $f$ be $\beta$-strongly convex with $L_1$-Lipschitz gradients, and let $h_1, h_2 > 0$ satisfy
\begin{equation}\label{eq:rho_NAG_cond}
\rho_{\mathrm{NAG}} := (1 - h_1\beta + 2h_1^2 L_1^2)(1+2h_2)^2 < 1.
\end{equation}
Then
\begin{equation}\label{eq:q_bound_NAGsc}
\sup_{k} \|q_k\| \leq \frac{h_1 \mu L_1(n+3)^{3/2}}{2} + \frac{(1+2h_2) h_1^2 L_1 \sigma}{\sqrt{t}} \sqrt{\frac{2}{1-\rho_{\mathrm{NAG}}}}\,.
\end{equation}
A sufficient condition for~\eqref{eq:rho_NAG_cond} is $h_1 \in \bigl(0, \frac{\beta}{2L_1^2}\bigr)$ and $h_2 \in \bigl(0, \frac{1}{2}\bigl(\frac{1}{\sqrt{1-h_1\beta+2h_1^2 L_1^2}} - 1\bigr)\bigr)$.
\end{lemma}

Both lemmas show that the averaged ZO iterates satisfy a perturbed version of their FO dynamics with bounded, controllable perturbations. The key structural difference is that, for NAG, the strong convexity contraction acts directly on the lookahead error, eliminating the momentum cross term that appears in the HB analysis. For HB, it is well known that the method locally converges exponentially fast~\cite{polyak1964some}; for NAG, global convergence with an accelerated rate is established in~\cite{nesterov2018lectures}. The following theorems provide a unified perspective that for both methods, the dimension dependence does not appear in the iteration complexity, instead, the convergence is to a neighbourhood of FO dynamics, where the size of the neighbourhood can be made arbitrarily small.

\begin{theorem}\label{th:HBsc}
Consider~\eqref{eq:unified_update} and~\eqref{eq:avg_HBNAG_perturbed} with $P=T$. Let $f$ be $\beta$-strongly convex with $L_1$-Lipschitz gradients, let the hypotheses of Lemma~\ref{lem:q_HBsc} be satisfied, and for $r > 0$ let $\mathcal{B}_r(z^e)$ denote a convergence neighbourhood around the fixed point $z^e$ of HB. Then there exist positive scalars $\bar{\mu}$ and $\bar{q}$, and $\bar{t} \in \mathbb{N}$ such that, for any $\mu \in (0, \bar{\mu})$ and $t \geq \bar{t}$ satisfying $\sup_k \|q_k\| \leq \bar{q}$ via the bound in Lemma~\ref{lem:q_HBsc}, ZO-HB has the same exponential decay rate as HB and converges to $\mathcal{B}_r(z^e)$, and there exists a class-$\mathcal{K}$ function $\alpha$ such that
\begin{equation}\label{eq:r_HBsc}
r = \alpha\Bigl(\frac{h_1 \mu L_1(n+3)^{3/2}}{2} + \frac{h_1^2 L_1 \sigma}{\sqrt{t}} \sqrt{\frac{2}{1-\rho_{\mathrm{HB}}}}\Bigr).
\end{equation}
\end{theorem}

\begin{theorem}\label{th:NAGsc}
Consider~\eqref{eq:unified_update} and~\eqref{eq:avg_HBNAG_perturbed} with $P=TA$. Let $f$ be $\beta$-strongly convex with $L_1$-Lipschitz gradients, let the hypotheses of Lemma~\ref{lem:q_NAGsc} be satisfied, and for $r > 0$ let $\mathcal{B}_r(z^e)$ denote a convergence neighbourhood around the fixed point $z^e$ of NAG. Then there exist $\bar{\mu}, \bar{q} > 0$ and $\bar{t} \in \mathbb{N}$ such that, for any $\mu \in (0, \bar{\mu})$ and $t \geq \bar{t}$ satisfying $\sup_k E\|q_k\| \leq \bar{q}$ via the bound in Lemma~\ref{lem:q_NAGsc}, ZO-NAG has the same decay rate as NAG and converges to $\mathcal{B}_r(z^e)$, and there exists a class-$\mathcal{K}$ function $\alpha$ such that
\begin{equation}\label{eq:r_NAGsc}
r = \alpha\Bigl(\frac{h_1 \mu L_1(n+3)^{3/2}}{2} + \frac{(1+2h_2) h_1^2 L_1 \sigma}{\sqrt{t}} \sqrt{\frac{2}{1-\rho_{\mathrm{NAG}}}}\Bigr).
\end{equation}
\end{theorem}

\subsection{$L_2$ regularisation}\label{sec:L2reg}

In the preceding sections, we showed that strong convexity is sufficient to guarantee that the averaged ZO dynamics can be written as a bounded perturbation of the FO dynamics. When the objective $f$ is not strongly convex, one can introduce an $L_2$ regularisation term to induce the necessary contraction to recover the same perturbation framework, even when $f$ itself is not strongly convex.


We therefore consider, in place of~\eqref{eq:main}, the regularised problem
\begin{align}\label{eq:reg_main}
    \min_{x\in\mathbb{R}^n}\; f(x) + \lambda\|x\|^2,
\end{align}
where $\lambda>0$ is the regularisation parameter. We focus on the GD case; analogous results for HB and NAG are discussed in Remark~\ref{rem:reg_HBNAG}. Setting $\lambda = \frac{1-c}{2h}$, where $h>0$ is the step size and $c\in [0,1)$ (thus $\lambda$ can be arbitrary small but not larger than $\frac{1}{2h}$), the GD update~\eqref{eq:GD_alg} applied to~\eqref{eq:reg_main} becomes
\begin{align}\label{eq:GD_alg_reg}
    z_{k+1} = cz_k -h \nabla f(z_k),
\end{align}
the ZO-GD update~\eqref{eq:ZOGD_alg} becomes
\begin{align}\label{eq:ZOGD_alg_reg}
    \tilde{z}_{k+1} = c\tilde{z}_k -h g_\mu(\tilde{z}_k),
\end{align}
and the averaged ZO update is
\begin{align}\label{eq:ZOGD_avg_reg}
\bar{z}_{k+1} = c\bar{z}_k -h \nabla f(\bar{z}_k) + q_k,
\end{align}
where
\begin{align}\label{eq:gen_q_ZOGD_reg}
   q_k = h \bigl(\nabla f(\bar{z}_k) - E_{U_k}[\nabla f_\mu(\tilde{z}_k)]\bigr) = h \bigl(\nabla f(\bar{z}_k)- \nabla f_\mu(\bar{z}_k)\bigr) + h\bigl(\nabla f_\mu(\bar{z}_k) - E_{U_k}[\nabla f_\mu(\tilde{z}_k)]\bigr).
\end{align}

\begin{lemma}\label{lem:q_GDreg}
    Consider~\eqref{eq:GD_alg_reg},~\eqref{eq:ZOGD_avg_reg}, and~\eqref{eq:gen_q_ZOGD_reg}. Let $f$ be a possibly non-strongly-convex function with Lipschitz gradients with constant $L_1>0$, and let $h\in\bigl(0,\frac{\sqrt{2(1+c^2)}-2c}{2L_1}\bigr)$. Then
    \begin{align}\label{eq:q_bound_GD_reg}
\sup_k \|q_k\| \leq h \frac{\mu}{2}L_1(n+3)^{3/2}+h L_1 \sqrt{\frac{2h^2\sigma^2}{t(1-c^2-4cL_1h-2L_1^2h^2)}}.
\end{align}
\end{lemma}


Lemma~\ref{lem:q_GDreg} shows that the contraction induced by the regularisation term ensures the averaged ZO-GD iterates satisfy a perturbed version of the GD dynamics, where the perturbations are bounded and controllable through the algorithm parameters. The following theorem establishes that, whenever GD converges to a fixed point, ZO-GD inherits the same convergence rate with no extra dimension dependence in the iteration complexity; the dimension affects only the size of the convergence neighbourhood.

\begin{theorem}\label{th:GD_reg}
Consider~\eqref{eq:GD_alg_reg} and~\eqref{eq:ZOGD_avg_reg}. Let $f$ be a possibly non-strongly-convex function with Lipschitz gradients with constant $L_1>0$, and let $h\in\bigl(0,\frac{\sqrt{2(1+c^2)}-2c}{2L_1}\bigr),$ and for $r>0$ let $\mathcal{B}_r(z^e)$ denote a convergence neighbourhood around the fixed point $z^e$ of GD. Suppose that GD applied to the regularised problem~\eqref{eq:reg_main} converges to a fixed point $z^e$. Then there exist constants $\bar{\mu}, \bar{q} > 0$ and $\bar{t}\in\mathbb{N}$ (depending on $h$, $c$, $L_1$, and the Lyapunov function of the GD dynamics) such that, for any $\mu \in (0, \bar{\mu})$ and $t \geq \bar{t}$ satisfying $\sup_k \|q_k\| \leq \bar{q}$ via the bound in Lemma~\ref{lem:q_GDreg}, ZO-GD has the same decay rate as GD and converges to $\mathcal{B}_r(z^e)$ and there exists a class-$\mathcal{K}$ function $\alpha$ such that
\begin{align}\label{eq:r_GD_reg}
r=\alpha\left( h \frac{\mu}{2}L_1(n+3)^{3/2}+h L_1 \sqrt{\frac{2h^2\sigma^2}{t(1-c^2-4cL_1h-2L_1^2h^2)}} \right).
\end{align}
\end{theorem}

The admissible parameter range in Lemma~\ref{lem:q_GDreg} implies that the regularised objective is strongly convex  (Remark~\ref{rem:reg_sc}). Nevertheless, the framework provides a systematic way to apply ISS to originally non-convex objectives with explicit control over the trade-off between proximity to the original minimiser and neighbourhood size. Next, we illustrate the theoretical findings via numerical examples.

\section{Numerical examples}\label{sec:exps}
We illustrate the main results through two settings: a strongly convex quadratic (Section~\ref{sec:quad_exp}) and a non-convex neural network with $L_2$ regularisation (Section~\ref{sec:nn_exp}). Extended parameter studies are provided in Appendices~\ref{app:quad_exp} and~\ref{app:nn_exp}.

\subsection{Quadratic objective}\label{sec:quad_exp}
We consider $\min_{x\in\mathbb{R}^n} x^\top Ax$, where $n=1000$ and $A$ is positive definite with condition number $100$ and maximum eigenvalue $100$, comparing GD and ZO-GD (averaged over 5 runs). Figure~\ref{fig:param_lr} shows the effect of varying $h \in \{10^{-4},10^{-5},10^{-7}\}$ with $\mu = 10^{-6}$ and $t=1$ fixed. When $h$ is too large, the perturbation exceeds the admissible bound of Theorem~\ref{thm:ISS_exp} and ZO-GD diverges. As $h$ decreases, ZO-GD converges to a progressively smaller neighbourhood of the GD trajectory, with no extra dimension dependence. Analogous sweeps over $\mu$ and $t$, as well as parameter compensation experiments, are presented in Appendix~\ref{app:quad_exp}. In all cases, the observations are consistent with the bound in Lemma~\ref{lem:q_GDsc}: $h$ and $\mu$ control the smoothing bias, while $h$ and $t$ control the variance term.

\begin{figure}[htb]
	\centering \includegraphics[width=0.9\textwidth]{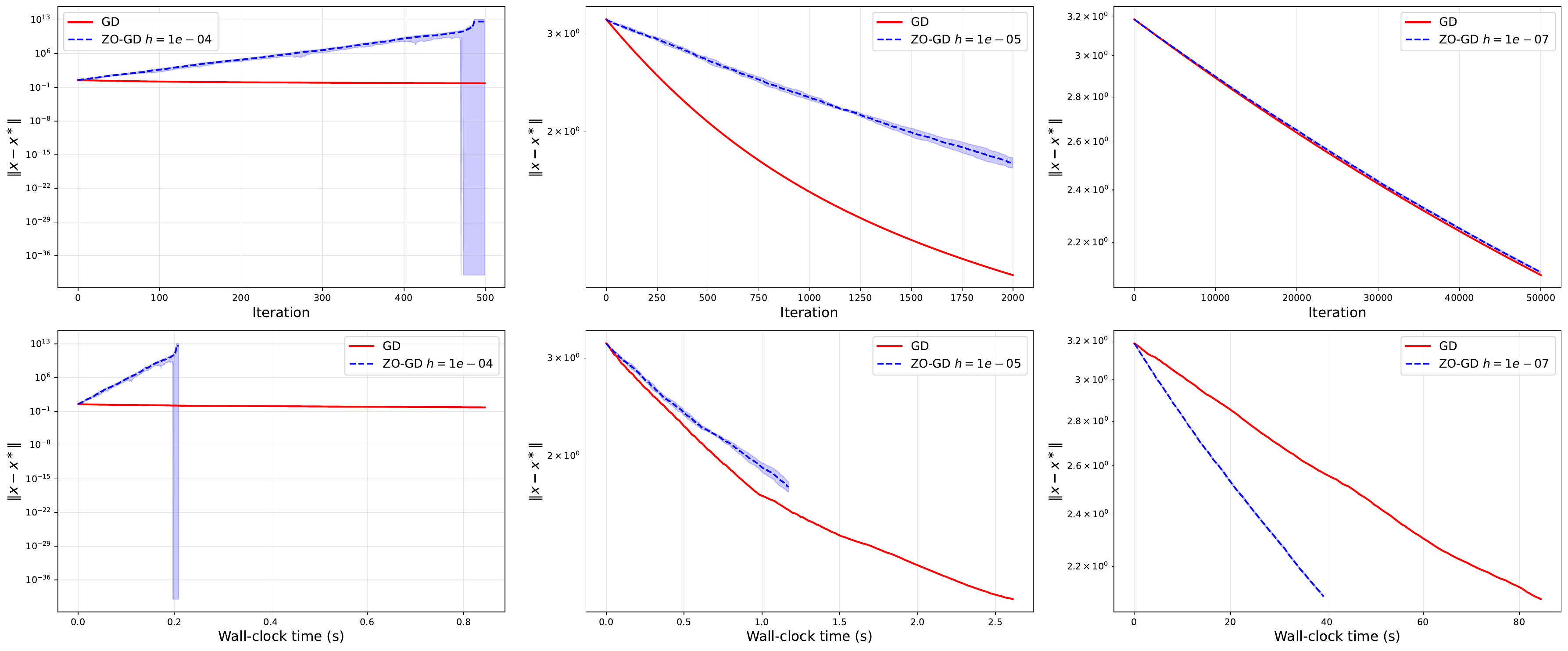}
	\caption{Quadratic objective ($n=1000$): effect of the step size $h$ on GD (solid) and ZO-GD (dashed). Reducing $h$ shrinks the perturbation and brings ZO-GD closer to GD.}\label{fig:param_lr}
\end{figure}

\subsection{Binary classification with a neural network}\label{sec:nn_exp}
We consider binary classification on MNIST (digits 0 vs 1) using a two-layer fully connected network with ReLU activations and sigmoid output ($d=784$, $H=128$ (hidden layer neurons), $n=100{,}609$ parameters). We minimise the $L_2$-regularised logistic loss with $\lambda=10^{-2}$ and $M=2000$ training samples. Since the logistic loss composed with a neural network is non-convex, this example falls in the framework of Section~\ref{sec:L2reg}.

\paragraph{Convergence comparison.} Figure~\ref{fig:nn_convergence} shows the training loss for all six methods (GD, HB, NAG and their ZO counterparts). Shaded regions indicate the standard deviation across 3 independent runs, and the small width confirms a small perturbation. We set $h=1\times10^{-3}$ for GD and ZO-GD, $h_1=1\times10^{-3}$, $h_2=5\times10^{-4}$ for HB/NAG and their ZO counterparts with smoothing parameter $\mu=10^{-7}$ and number of sampled directions $t=7$. The ZO trajectories closely track their FO counterparts, and all methods reach $99.55\%$ training accuracy and $99.40\%$ test accuracy. It takes approximately 4 times longer to run the ZO methods compared to their FO counterpart. Parameter sensitivity experiments confirming the same phenomena as in the quadratic case are provided in Appendix~\ref{app:nn_exp}.

\begin{figure}[htb]
	\centering \includegraphics[width=0.9\textwidth]{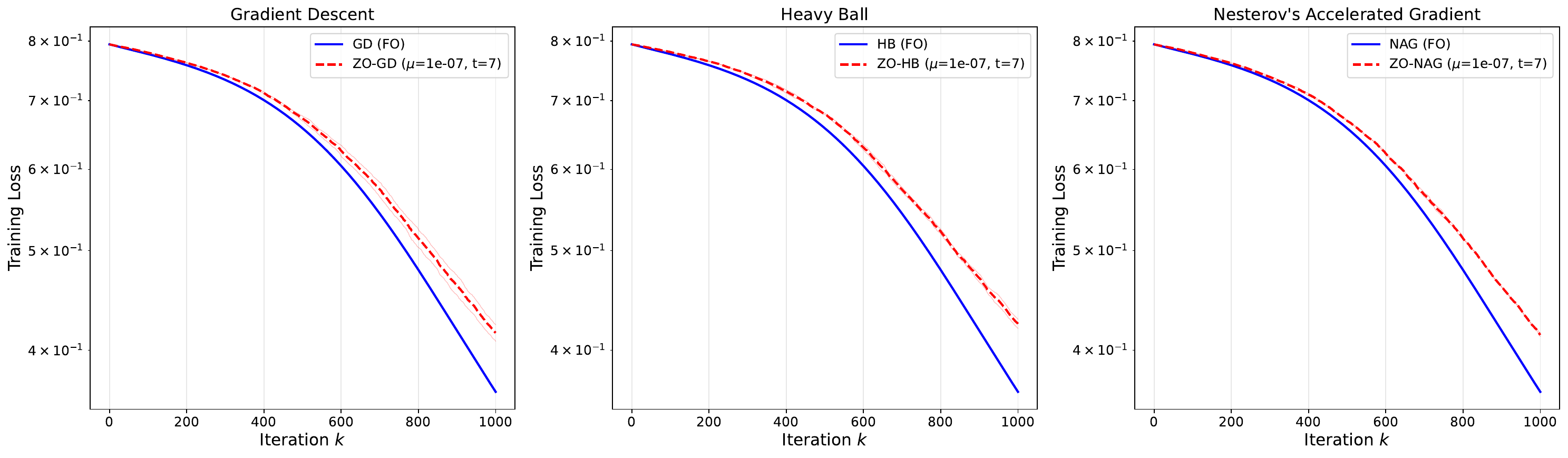}
	\caption{Training loss for GD, HB, NAG and their ZO counterparts on a two-layer neural network ($n=100{,}609$). Solid: FO; dashed: ZO. All methods achieve $99.55\%$ train / $99.40\%$ test accuracy.}\label{fig:nn_convergence}
\end{figure}

\paragraph{Dimension scaling.} To directly validate the absence of extra dimension dependence, we vary $H\in\{16,32,64,128,256\}$ ($n$ from ${\approx}12{,}700$ to ${\approx}201{,}000$) and run GD and ZO-GD with identical parameters for 1000 iterations. We set $h = 5\times10^{-4},$ $t=10,$ and $\mu=10^{-8}.$ Figure~\ref{fig:nn_scaling}~ shows that ZO-GD tracks GD at the same rate across all dimensions.  Table~\ref{tab:dim_scaling} reports iteration counts to reach within $20\%$ of the FO final loss: the ratio ZO/FO stays close to $1$. Moreover, Table~\ref{tab:dim_scaling} confirms the per-iteration time ratio remains approximately constant.

\begin{figure}[htb]
	\centering \includegraphics[width=0.45\textwidth]{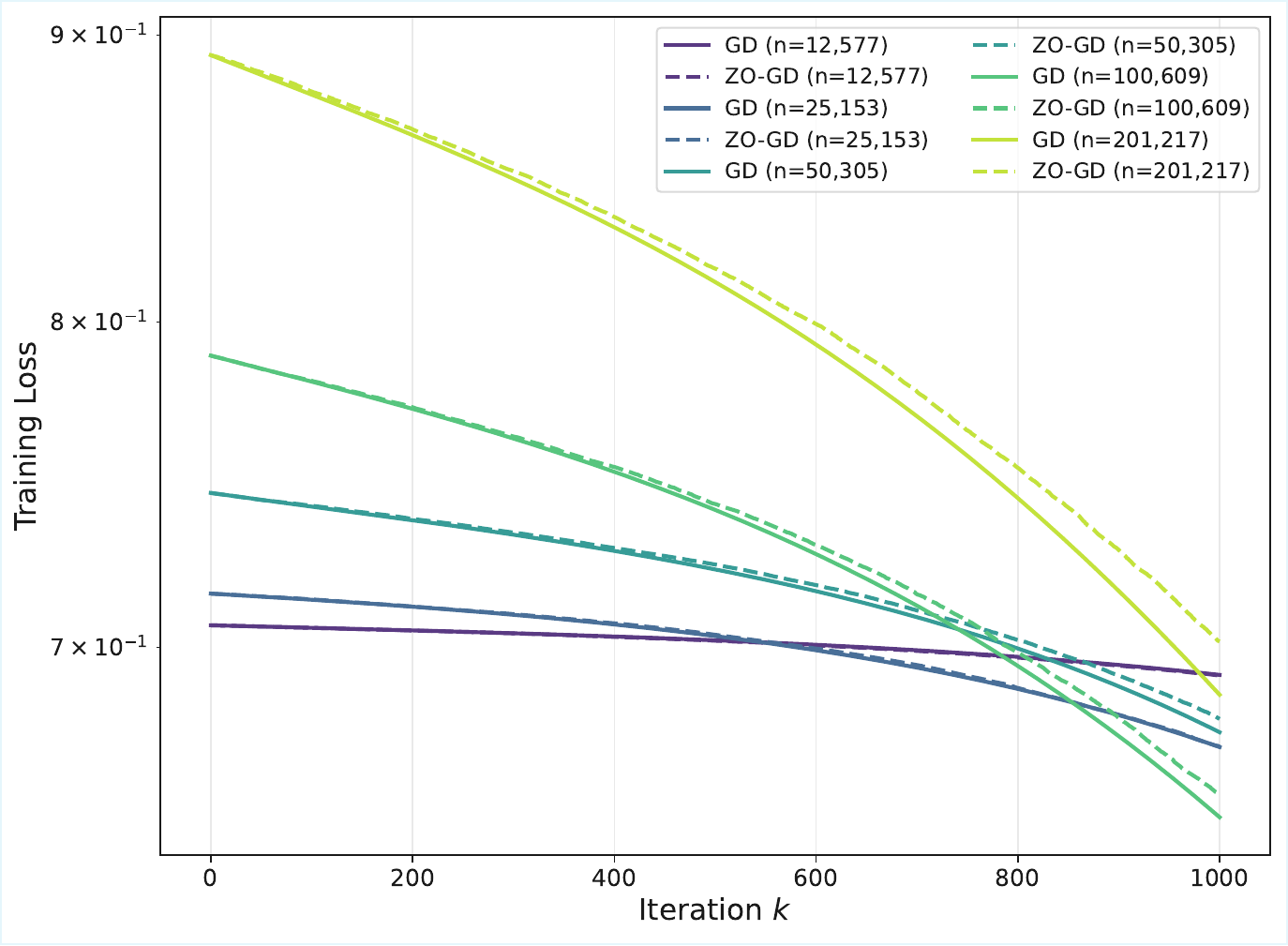}
	\caption{Dimension scaling: 
    Training loss for GD (solid) and ZO-GD (dashed) across five network sizes.
    }\label{fig:nn_scaling}
\end{figure}


\begin{table}[htb]
\centering
\caption{Dimension scaling: iterations to reach within $20\%$ of FO final loss.}\label{tab:dim_scaling}
\begin{tabular}{@{}cccccccc@{}}
\toprule
$H$ & $n$ & FO final & ZO final  & Iter ZO/FO & Time ZO/FO \\
\midrule
16  & 12,577  & 0.6922 & 0.6920   & 0.99  & 5.5  \\
32  & 25,153  & 0.6720 & 0.6721   & 1  & 4.9  \\
64  & 50,305  & 0.6761 & 0.6798  & 1.03 & 5.4 \\
128 & 100,609 & 0.6529 & 0.6584 &1.03 & 5.7 \\
256 & 201,217 & 0.6866 & 0.7017 & 1.05 & 4.5 \\
\bottomrule
\end{tabular}
\end{table}

\section{Conclusion and future directions}\label{sec:conc}

In this work, we revisited the commonly held belief that ZO methods inherently suffer from additional dimension dependence in their iteration complexity compared to their FO counterparts. By adopting a dynamical systems perspective, we showed that, under suitable conditions, the averaged dynamics of ZO algorithms can be interpreted as perturbed versions of their FO counterparts with bounded, controllable perturbations. Leveraging ISS arguments, we proved that ZO methods inherit the same convergence rates as FO methods in expectation, converging to a neighbourhood of the FO fixed point whose size can be made arbitrarily small by tuning the algorithm parameters. We note that while the iteration complexity is dimension-free, achieving a target neighbourhood radius may require parameter choices (e.g., $\mu$ or $t$) that depend on problem parameters, including the variance, so the total function evaluation cost may retain a dimension dependence in the worst case scenario. These findings suggest that the practical efficiency of ZO methods may be closer to that of FO methods than previously understood, particularly when gradient information is unavailable or expensive.

Promising directions for future work include extending the analysis to broader classes of non-convex and non-smooth problems, establishing high-probability guarantees rather than results in expectation, and exploring adaptive schemes for tuning the smoothing and step-size parameters to optimise the trade-off between convergence speed and neighbourhood size.


\bibliography{NIPS_2026}
\bibliographystyle{alpha}

\appendix
\section{Related works and Basic definitions}\label{app:def}

In this section, we provide a detailed comparison of our results with prior work on zeroth-order optimisation, followed by the basic definitions used throughout the paper.

\subsection{Comparison with prior work}

Table~\ref{tab:comparison} summarises the convergence guarantees of representative ZO methods and highlights how our framework differs from existing analyses. In all prior works, the iteration complexity carries an explicit factor depending on the problem dimension $n$ (or an effective dimension $n_{\mathrm{eff}}$), and convergence is to the exact optimiser. By contrast, our ISS-based framework eliminates the dimension factor from the iteration complexity entirely; the dimension dependence instead appears in the radius of the convergence neighbourhood, which can be made arbitrarily small by tuning the algorithm parameters $\mu$, $h$, and $t$. Moreover, our framework applies uniformly to ZO-GD, ZO-HB, and ZO-NAG, whereas each prior result requires a separate, algorithm-specific analysis. We further note that in prior analyses such as \cite{nesterov_random_2017}, the prescribed step size for ZO-GD scales as $O(1/n)$, which is not consistent with standard FO gradient descent. By contrast, our  framework permits the same step size as the FO method, with no dependence on the dimension (see Remark~\ref{rem:dim_h}).

\begin{table}[t]
\footnotesize
\centering
\caption{Comparison of convergence guarantees for ZO methods.}
\label{tab:comparison}
\renewcommand{\arraystretch}{1.3}
\setlength{\tabcolsep}{4pt}
\small
\begin{tabular}{@{}lcccc@{}}
\toprule
\textbf{Reference} & \textbf{Assumptions on $f$} & \textbf{Iter. Dim. dependence}& \textbf{queries per iter.} & \textbf{Convergence target} \\
\midrule
\cite{nesterov_random_2017}
  & convex / s.c.\
  & $O(n)$
  & 2
  & Minimiser\\
\cite{duchi2015optimal}
  & convex
  & $O(n)$
  & 2
  & Minimiser\\
\cite{ghadimi2013stochastic}
  & nonconvex
  & $O(n)$
  & 2$^\ast$
  & $\|\nabla f\|^2 \leq \varepsilon$\\
\cite{yue2023zeroth}
  & \makecell[c]{convex / s.c;\\
  Hessian regularity}
  & $O(n_{\mathrm{eff}})$
  & 2
  & Minimiser \\
\cite{wang2018stochastic}
  & \makecell[c]{convex;\\sparse gradients}
  & $O(s \log n)$
  & 2
  & Minimiser\\
\midrule
\textbf{This work}
  & \makecell[c]{s.c.\ (Thms.~3--5);\\n.c + $L_2$ reg.\ (Thm.~6)}
  & \textbf{No explicit}
   & $2t$
  & \makecell[c]{Neighbourhood of \\FO fixed point (can be \\made arbitrarily small)}\\
\bottomrule

\end{tabular}

\vspace{0.5em}

{\footnotesize s.c.\ = strongly convex; n.c.\ = non-convex; $2^\ast:$ using diminishing step sizes.}
\end{table}

\subsection{Basic definitions}

We now recall the standard regularity and convexity conditions assumed throughout the paper.

\begin{definition}[Lipschitz continuity and gradients]\label{def:lip}
	A continuous function $f:\mathbb{R}^n \rightarrow \mathbb{R}$ is globally Lipschitz with constant $L_0>0$ if $\| f(x)- f(y)\|\leq L_0\|x-y\|$ for all $x,y\in \mathbb{R}^n$. If $f$ is $C^1$, the gradient of $f$ is globally Lipschitz with constant $L_1>0$ if
	\begin{align}
		\|\nabla f(x)-\nabla f(y)\|\leq L_1\|x-y\|, \qquad \forall \ x,y\in \mathbb{R}^n. \label{eq:L_gradient}
	\end{align}
\end{definition}

\begin{definition}[Strong convexity]\label{def:strong_convex}
Let $f:\mathbb{R}^n \rightarrow \mathbb{R}$ be a $C^1$
function. 
Then $f$ is said to be $\beta$-strongly convex with constant $\beta>0$ if there exists a constant $\beta>0$ such that
\begin{align}
f(y) \geq f(x) + \nabla f(x)^\top (y-x) + \frac{\beta}{2}\|x-y\|^2, \qquad \forall \ x,y\in \mathbb{R}^n. 
\end{align}
\end{definition}

\section{Additional details and proofs of Section~\ref{sec:per}}\label{app:per}
In this section, we present additional explanations and proofs of the results in Section~\ref{sec:per}.

\begin{lemma}[\mbox{\cite[Lem. 10]{kellett2014compendium}}]\label{lem:addk}
Let $\alpha:\mathbb{R}_{\ge 0}\to\mathbb{R}_{\ge 0}$ be a class-$\mathcal{K}$ function. Then for any $a,b\ge0$ and any $\varepsilon>0$,
\begin{equation}
\alpha(a+b)
\le
\alpha((1+\varepsilon)a)
+
\alpha\!\left((1+\tfrac{1}{\varepsilon})b\right).
\end{equation}
\end{lemma}

\begin{proof}[Proof of Theorem~\ref{thm:ISS_exp}]
Let $R>0$ be such that \eqref{eq:sandwich}--\eqref{eq:decrease} hold on $\mathcal{B}_R(z^e)$. Since $V$ is locally Lipschitz on $\mathcal{B}_R(z^e)$, there exist constants $L_V>0$ satisfying
\[
|V(x)-V(y)|\le L_V\|x-y\|
\]
for all $x,y\in\mathcal{B}_R(z^e)$. Write $\lambda:=1-c\in(0,1)$ and $\delta:=\sup_{j\ge 0}\|q_j\|$.
Let $\bar{r},\bar{q}>0$, let $z_0\in \mathcal{B}_{\bar{r}}(z^e)$, let $\{z_k\}_{k\in \N}$ be defined through $z_{k+1}=w(z_k)+q_k$ for $k\in \N$ and $\{q_k\}_{k\in \N} \subset \mathcal{B}_{\bar{q}}(z^e)$, and let $\bar{r}$, $\bar{q}$ be defined such that $\{z_k\}_{k\in \N} \subset \mathcal{B}_{\bar{r}}(z^e)$. (The existence of $\bar{r},\bar{q}>0$ with this property will be shown later.) 
Using the Lipschitz property of $V$ and \eqref{eq:decrease}, for
the perturbed successor $z_{k+1}=w(z_k)+q_k$, the estimate
\begin{equation}\label{eq:one_step}
V(z_{k+1})
= V\bigl(w(z_k)+q_k\bigr)
\le V(w(z_k)) + L_V\|q_k\|
\le \lambda\,V(z_k) + L_V\|q_k\|
\end{equation}
is obtained. Iterating~\eqref{eq:one_step} from $k=0$ yields
\[
V(z_k)
\le \lambda^k V(z_0) + L_V\sum_{j=0}^{k-1}\lambda^{k-1-j}\|q_j\|
\le \lambda^k V(z_0) + L_V\delta\sum_{j=0}^{k-1}\lambda^j
\le \lambda^k V(z_0) + \frac{L_V}{c}\,\delta,
\]
where the geometric series was bounded by $\sum_{j=0}^{\infty}\lambda^j = 1/c$. Applying $c_1\|z_k-z^e\|^2\leq V(z_k)$ and the sandwich bound $V(z_0)\le c_2\|z_0-z^e\|^2$, we obtain
\[
\|z_k-z^e\|^2
\le \lambda^kc_1^{-1}\,c_2\|z_0-z^e\|^2+\tfrac{L_V}{cc_1}\,\delta
\]
or
\[
\|z_k-z^e\|
\le \sqrt{\frac{c_2}{c_1}}\|z_0-z^e\|\lambda^{\frac{k}{2}}
   +\gamma\!\Bigl(\sup_{0\le j\le k}\|q_j\|\Bigr),
\qquad k\ge 0,
\]
 where $\gamma(\delta) = \sqrt{\frac{L_v}{cc_1}}\sqrt{\delta},$ which is a $\KK$ function.

It remains to choose $\bar r>0$ and $\bar q>0$ so that
$\{z_k\}_{k\in \N} \subset \mathcal{B}_R(z^e)$. 
For $z\in \R^n$ such that $V(z) \leq \ell =  c_1R^2$, \eqref{eq:sandwich} implies that 
$\|z-z^e\|^2 \leq R^2$, i.e.,  $z \in \mathcal{B}_R(z^e)$, and we can conclude $\{z\in \R^n| \ V(z) \leq  c_1R^2 \} \subset \mathcal{B}_R(z^e)$. 
From~\eqref{eq:one_step}, the set $\{z\in \R^n| \ V(z) \leq  \ell \}$
is forward-invariant for the perturbed system whenever $\lambda\,\ell+L_V\bar q\le\ell$, i.e.\ whenever $\bar q\le c\,\ell/L_V$. Setting
\[
\bar r:=\sqrt{\frac{\ell}{c_2}}, \qquad \bar q:=\frac{c\,\ell}{L_V}
\]
ensures that $\|z_0-z^e\|\le\bar r$ implies $V(z_0)\le\ell$, and the sublevel set $\{z\in \R^n| \ V(z) \leq  \ell \}$ is forward-invariant under the perturbed dynamics, so the Lipschitz bounds and the decrease condition remain valid for all $k\ge 0$. This completes the proof.
\end{proof}
\begin{proof}[Proof of Lemma~\ref{lem:comparison}]

\textbf{Item~\textup{(i)}:} Since $\varphi$ is continuous with $\varphi(0)=0$ and $\varphi(s)<s$ for $s>0$, the iterates $\varphi^{(k)}(s):=\varphi\circ\cdots\circ\varphi(s)$ ($k$ times) satisfy $\varphi^{(k)}(s)\downarrow 0$ as $k\to\infty$ for each fixed $s>0$. Define $\hat\beta(s,k):=\varphi^{(k)}(s)$. This is continuous, non-decreasing in $s$, decreasing to $0$ in $k$, and $\hat\beta(0,k)=0$; so it can be upper-bounded by a class-$\KL$ function, which we denote $\hat\beta$. This gives (i).

\textbf{Item~\textup{(ii)}:} Fix $\varepsilon>0$. Since $\varphi(\varepsilon)<\varepsilon$, set $\mu:=\varepsilon-\varphi(\varepsilon)>0$ and take $\bar d(\varepsilon):=\mu/2$. If $d\le\bar d(\varepsilon)$ and $v_k\le\varepsilon$, then $v_{k+1}\le\varphi(\varepsilon)+d\le\varphi(\varepsilon)+\mu/2=\varepsilon-\mu/2<\varepsilon$. This in particular implies that $\varphi(s)\in [0,\varepsilon]$ for all $s\in [0,\varepsilon]$.
Meanwhile, if $v_k>\varepsilon$ we have $v_{k+1}\le\varphi(v_k)+d<v_k-\mu+d\le v_k-\mu/2$ (using $\mu\le s-\varphi(s)$ for $s\ge\varepsilon$ by continuity and compactness on $[\varepsilon,v_0]$). Hence, for $v_k>\varepsilon$, 
$v_k$ decreases by at least $\mu/2$ per step and there exists $K\in \N$ such that $v_K\in [0,\varepsilon]$.

\textbf{Item~\textup{(iii)}:} Fix $R>0$ throughout and let $\hat\beta\in\KL$ be the function from part~(i).
Since $\mathrm{Id}-\varphi : [0,\infty)\rightarrow [0,b] $ is class-$\KK$ and $d\leq b$, its inverse is strictly increasing and is defined at $d$.

Define
\[
\sigma(d) := (\mathrm{Id}-\varphi)^{-1}(d),
\]
so that $\varphi(\sigma(d))+d = \sigma(d)$.
We record two facts, both following directly from $\mathrm{Id}-\varphi$ being strictly increasing. We will analyse $v_{k+1},$ for the cases where $v_k\in[0,\sigma(d)]$ or $v_k\in(\sigma(d),\infty).$

\emph{Forward invariance} of $[0,\sigma(d)]$ under the mapping $\phi(\cdot)+d$: if $v_k\le\sigma(d)$, then
$v_{k+1}\le\varphi(\sigma(d))+d=\sigma(d)$, so $[0,\sigma(d)]$ is forward-invariant.
 
\emph{Strict decrease} of $v_k$ on $(\sigma(d),\infty)$:
if $v_k>\sigma(d)$, then
$v_k-\varphi(v_k)>\sigma(d)-\varphi(\sigma(d))=d$, hence $\phi(v_k)+d<v_k$ and
$v_{k+1}\le\varphi(v_k)+d<v_k$.
 
Combining these two properties, the sequence $\{v_k\}$ is bounded above by $\max(v_0,\sigma(d))$ for all $k\ge 0$: it cannot increase above $v_0$ once it starts decreasing (which happens immediately if $v_0>\sigma(d)$), and it is trapped once it enters $[0,\sigma(d)]$. In particular, for $v_0\in[0,R]$:
\begin{equation}\label{eq:uniform}
v_k \le \max\bigl(R,\,\sigma(d)\bigr) \quad\text{for all } k\ge 0.
\end{equation}
 
We claim that for all $v_0\in[0,R]$, $d\ge 0$, and $k\ge 0$,
\begin{equation}\label{eq:goal}
v_k \le \hat\beta(v_0,k) + \eta_R(d),
\end{equation}
for some function $\eta_R$ with $\eta_R(0)=0$ that depends only on $d$ (and $R$), not on $v_0$ or $k$.
To see this, write $v_k = \hat\beta(v_0,k) + e_k$ where $e_k := v_k - \hat\beta(v_0,k)$ is the excess of the perturbed trajectory over the unperturbed bound. We need to show that $e_k$ is bounded by a function of $d$ 
(uniformly over $v_0\in[0,R]$ and $k\ge 0$).
Consider the unperturbed comparison sequence $u_0=v_0$, $u_{k+1}=\varphi(u_k)$, so $u_k=\hat\beta(v_0,k)$ by part~(i). Both sequences satisfy:
\begin{alignat*}{2}
v_{k+1} &\le \varphi(v_k) + d, &\qquad v_0&=s\in[0,R],\\
u_{k+1} &= \varphi(u_k), &\qquad u_0&=s.
\end{alignat*}
Subtracting and using $\varphi(v_k)-\varphi(u_k)\le v_k - u_k$ (which follows from $\mathrm{Id}-\varphi$ being non-decreasing, so $v_k-\varphi(v_k)\ge u_k-\varphi(u_k)$ when $v_k\ge u_k$):
\[
v_{k+1}-u_{k+1} \le (v_k - u_k) + d \quad\text{whenever } v_k\ge u_k.
\]
Since $v_0-u_0=0$, iterating this inequality gives $v_k - u_k \le kd$ whenever $v_k\ge u_k$ persists for all steps up to $k$.
 
At first glance, $kd$ grows without bound. But recall that $v_k$ eventually enters and remains in $[0,\sigma(d)+\delta]$ for any fixed $\delta>0$. Let $K^*$ be the first time $v_k\le\sigma(d)+\delta$. The strict decrease above $\sigma(d)+\delta$ gives a minimum decrease per step: for $v_k\in[\sigma(d)+\delta,\max(R,\sigma(d)+\delta)]$, we have
\[
v_k - v_{k+1} \ge v_k - \varphi(v_k) - d \ge (\mathrm{Id}-\varphi)(v_k) - d.
\]
On the compact set $[\sigma(d)+\delta,\max(R,\sigma(d)+\delta)],$ where $\delta>0$ can be arbitrary small (fixed and bounded away from zero), the continuous function $(\mathrm{Id}-\varphi)(s)-d$ is strictly positive (being zero only at $s=\sigma(d)$). Hence, when $R>\sigma(d)+\delta,$ the entry time satisfies the existence of $K^\ast$ to have $v_{K^\ast}\le\sigma(d)+\delta$ (and $K^*=0$ otherwise). In particular, $K^*$ depends only on $R$ and $d$ (not on $v_0$ beyond the constraint $v_0\le R$).
Now we bound the excess for all $k$:
\begin{itemize}
\item For $k\le K^*$: using $v_k\le\max(R,\sigma(d)+\delta)$ from~\eqref{eq:uniform}
and $\hat\beta(v_0,k)\ge 0$, we get $e_k\le\max(R,\sigma(d)+\delta)$. But more
precisely, by the iterated bound above: $e_k = v_k - u_k \le K^* d$. 
\item For
$k> K^*$: $v_k\le\sigma(d)+\delta$ and $\hat\beta(v_0,k)\ge 0$, so $e_k\le\sigma(d)+\delta$.
\end{itemize}
Therefore, for all $k\ge 0$:
\[
e_k \le \max\bigl(K^* d,\;\sigma(d)+\delta\bigr) 
=: \eta_{R,\delta}(d).
\]
Since $\delta>0$ is arbitrary, for any $d>0$ we can 
choose $\delta = d$ (or any class-$\mathcal{K}$ function 
of $d$), giving
\[
e_k \le \max\bigl(K^* d,\;\sigma(d)+d\bigr) 
=: \eta_R(d).
\]
Since $\sigma(0)=0$ and $K^*=0$ when $d=0$, we have 
$\eta_R(0)=0$. The function $\eta_R$ is non-negative, 
finite, nondecreasing in $d$, and satisfies 
$\eta_R(0)=0$. By regularisation, there exists 
$\hat\gamma_R\in\mathcal{K}$ with 
$\eta_R(d)\le\hat\gamma_R(d)$ for all $d\ge 0$.
 
Substituting into~\eqref{eq:goal} and using $u_k = \hat\beta(v_0,k)$:
\[
v_k \le \hat\beta(v_0,k) + \hat\gamma_R(d), \qquad k\ge 0,
\]
for all $v_0\in[0,R]$ and $d\ge 0$.
\end{proof}

\begin{lemma}[\mbox{\cite[Lemma B.1]{jiang2001input}}]\label{lem:i-k}
    For any $\KKinf$ function $\alpha,$ there exists a $\KKinf$ function $\hat{\alpha}$ such that the following holds:
    \begin{itemize}
        \item $\hat{\alpha}(r)\leq\alpha(r)$ for all $r\geq0$; and,
        \item $\mathrm{Id} -\hat{\alpha}\in\KK.$
    \end{itemize}
\end{lemma}

\begin{proof}[Proof of the Theorem~\ref{thm:ISS_asymp}]
Let $R>0$ be chosen so that~\eqref{eq:sandwich}--\eqref{eq:decrease} hold on $\mathcal{B}_R(z^e)$, and let $L_V>0$ be a Lipschitz constant for $V$ on $\mathcal{B}_R(z^e)$. Write $\delta:=\sup_{j\ge 0}\|q_j\|$. Provided the trajectory stays in $\mathcal{B}_R(z^e)$, the Lipschitz property of $V$ gives
\begin{equation}\label{eq:onestep}
V(z_{k+1})
= V(w(z_k)+q_k)
\le V(w(z_k)) + L_V\|q_k\|.
\end{equation}
Combining~\eqref{eq:onestep} with the decrease condition~\eqref{eq:decrease} and the upper sandwich bound~\eqref{eq:sandwich} yields
\begin{equation}\label{eq:Vrecursion}
V(z_{k+1}) \le V(z_k) - \alpha_3(\|z_k-z^e\|) + L_V\delta
\le V(z_k) - \alpha_3\!\bigl(\alpha_2^{-1}(V(z_k))\bigr) + L_V\delta,
\end{equation}
where the second inequality uses $\|z_k-z^e\|\ge\alpha_2^{-1}(V(z_k))$ from the right-hand side of~\eqref{eq:sandwich}. Define $\rho:=\alpha_3\circ\alpha_2^{-1}$, which is class-$\KKinf$, and set $v_k:=V(z_k)$. Then~\eqref{eq:Vrecursion} reads
\[
v_{k+1}\le v_k - \rho(v_k) + L_V\delta.
\]
Define $\varphi(s):=s-\rho(s)$. Since $\rho$ is class-$\KKinf,$ using Lemma~\ref{lem:i-k}, without loss of generality we can say $\varphi\in\KK$, thus $\varphi$ is continuous, $\varphi(0)=0$, and $\mathrm{Id}-\varphi\in\KKinf.$ Thus, we have 
\[
v_{k+1}\le\varphi(v_k)+L_V\delta,
\]
which is exactly the setting of Lemma~\ref{lem:comparison} with $d=L_V\delta$. Since Theorem~\ref{thm:ISS_asymp} is concerned with properties on a compact set, 
let $v_0\in[0,\bar{R}],$ where $\bar{R}>0$ such that $\{z:V(z)\leq\bar{R}\}\subset \mathcal{B}_R(z^e)$. Applying part~(iii) of the lemma gives
a class-$\KL$ function $\hat\beta$ and a class-$\KK$ function $\hat\gamma$ such that
\[
V(z_k)\le\hat\beta(V(z_0),\,k)+\hat\gamma_{\bar{R}}(L_V\delta).
\]
According to
the sandwich bounds \eqref{eq:sandwich1} the inequalities
$V(z_0)\le\alpha_2(\|z_0-z^e\|)$ and $\|z_k-z^e\|\le\alpha_1^{-1}(V(z_k))$ are satisfied. Using Lemma~\ref{lem:addk} with $\varepsilon >0$ fixed, we obtain
\[
\|z_k-z^e\|
\le \alpha_1^{-1}\!\bigl((1+\epsilon)\,\hat\beta(\alpha_2(\|z_0-z^e\|),\,k)\bigr)
  + \alpha_1^{-1}\!\bigl((1+\tfrac{1}{\epsilon})\hat\gamma_{\bar{R}}(L_V\delta)\bigr).
\]
With the definitions
\[
\beta(s,k):=\alpha_1^{-1}\!\bigl((1+\epsilon)\hat\beta(\alpha_2(s),\,k)\bigr), \qquad
\gamma(s):=\alpha_1^{-1}\!\bigl((1+\tfrac{1}{\epsilon})\hat\gamma_{\bar{R}}(L_V s)\bigr),
\]
we see that $\beta$ is class-$\KL$ (as a composition of class-$\KKinf$ and class-$\KL$ functions) and $\gamma$ is class-$\KK$, giving
\[
\|z_k-z^e\|\le\beta(\|z_0-z^e\|,\,k)+\gamma\!\Bigl(\sup_{0\le j\le k}\|q_j\|\Bigr).
\]
 

To complete the proof, we 
need trajectories
to remain in $\{z \in \R^n :\ V(z)\le\bar{R}\}\subset\mathcal{B}_R(z^e)$.
From the 
proof of Lemma~\ref{lem:comparison}(iii), we know that $v_k\le\max(v_0,\sigma(L_V\delta))$,
so forward invariance of $\{z \in \R^n :\ V(z)\le\bar{R}\}$ holds provided
$\sigma(L_V\bar{q})\le\bar{R}$, i.e., $\bar{q}\le\rho(\bar{R})/L_V$.
Setting
\[
\bar r:=\alpha_2^{-1}(\bar{R}), \qquad \bar{q}:=\frac{\rho(\bar{R})}{L_V},
\]
we have that $\|z_0-z^e\|\le\bar r$ implies $V(z_0)\le\bar{R}$, and
$\sup_k\|q_k\|\le\bar{q}$ guarantees $\sigma(L_V\bar{q})=\bar{R}$,
so the sublevel set $\{z \in \R^n :\ V(z)\le\bar{R}\}$ is forward-invariant under
the perturbed dynamics. Hence, the proof is complete.
\end{proof}

\begin{lemma}\label{lm:fmu_sc}
Let $f$ be $\beta$-strongly convex. Then $f_\mu$ is $\beta$-strongly convex.    
\end{lemma}
\begin{proof}
From the strong convexity of $f$, we have $f(y)\geq f(x) +\langle \nabla f(x),y-x \rangle + \tfrac{\beta}{2}\|y-x\|^2$. Replacing $x$ and $y$ with $x+\mu u$ and $y+\mu u$, respectively, and taking expectation with respect to $u$, we obtain
\[f_\mu(y)\geq f_\mu(x) +\langle \nabla f_\mu(x),y-x \rangle + \tfrac{\beta}{2}\|y-x\|^2,\]
which completes the proof.
\end{proof}

\section{Proof of lemmas and theorems of Section~\ref{sec:main}}\label{app:main}

\begin{proof}[Proof of Lemma~\ref{lem:q_GDsc}]
     Consider $q_k$ given in \eqref{eq:gen_q_ZOGD}, which satisfies
    \begin{align}\label{eq:q_ZOGD_norm}
        \|q_k\|\leq h\|\nabla f(\bar{z}_k)- \nabla f_\mu(\bar{z}_k)\|+h\|\nabla f_\mu(\bar{z}_k) - E_{U_k}[\nabla f_\mu(\tilde{z}_k)]\|.
    \end{align}
    Using \cite[Lemma 3]{nesterov_random_2017}, we know that \( \|\nabla f(\bar{z}_k)- \nabla f_\mu(\bar{z}_k)\|\leq \tfrac{\mu}{2}L_1(n+3)^{3/2}.\) Thus, the first term
    on the right-hand side
    \eqref{eq:q_ZOGD_norm} is bounded and can be manipulated through
    the choice of $h$ and $\mu.$ For the second term on the right-hand side
    of \eqref{eq:q_ZOGD_norm}, we have 
    \begin{align}
        \|E_{U_k}[\nabla f_\mu(\tilde{z}_k)]- \nabla f_\mu(\bar{z}_k)\| &\leq \|E_{U_k}[\nabla f_\mu(\tilde{z}_k)- \nabla f_\mu(\bar{z}_k)]\| 
        \leq E_{U_k}[\|\nabla f_\mu(\tilde{z}_k)- \nabla f_\mu(\bar{z}_k)\|]\notag\\
        &\leq L_1E_{U_k}[\|\tilde{z}_k- \bar{z}_k\|]
    \end{align}
    Now let $e_k = \tilde{z}_k- \bar{z}_k$, i.e., we need to find an upper bound on
    $E[\|e_k\|]$ or $E[\|e_k\|^2]$ to prove the assertion. 
Combining \eqref{eq:ZOGD_alg}, \eqref{eq:ZOGD_avg} and \eqref{eq:gen_q_ZOGD}, we
    know that
    \begin{align*}
      e_{k+1} &= e_k -h g_\mu(\tilde{z}_k) +h E_{U_k}[\nabla f_\mu(\tilde{z}_k)] 
      \\&=e_k - h (g_\mu(\tilde{z}_k)- \nabla f_\mu(\tilde{z}_k)) -h (\nabla f_\mu(\tilde{z}_k) - E_{U_k}[\nabla f_\mu(\tilde{z}_k)]).
    \end{align*}
    To proceed, we
    denote $v_k = g_\mu(\tilde{z}_k)- \nabla f_\mu(\tilde{z}_k)$ and $l_k =  \nabla f_\mu(\tilde{z}_k) - E_{U_k}[\nabla f_\mu(\tilde{z}_k)]$ and thus
    $e_{k+1} = e_k -h (v_k+l_k).$ Considering 
    Assumption~\ref{lm20},  the variance of the random oracle is bounded, and we have $E[\|v_k\|^2]\leq\frac{\sigma^2}{t}$. Moreover, for $l_k$
    we have
    \begin{align}\label{eq:l}
        E_{U_k}[\|l_k\|] &=E_{U_k}[\|\nabla f_\mu(\tilde{z}_k)-\nabla f_\mu(\bar{z}_k)+\nabla f_\mu(\bar{z}_k)- E_{U_k}[\nabla f_\mu(\tilde{z}_k)]\|]
        \notag\\&\leq E_{U_k}[\|\nabla f_\mu(\tilde{z}_k)-\nabla f_\mu(\bar{z}_k)\|] + E_{U_k}[\|\nabla f_\mu(\bar{z}_k)-\nabla f_\mu(\tilde{z}_k)\|]\notag\\&\leq2L_1E_{U_k}[\|e_k\|]
    \end{align}
    Thus we get 
    \begin{align}
        E[\|e_{k+1}\|] \leq (1+2L_1h)E[\|e_k\|] +h \frac{\sigma}{\sqrt{t}},
    \end{align}
which does not provide a
uniform bound on $E_{U_k}[\|\tilde{z}_k- \bar{z}_k\|]$  yet.
To proceed, we use the strong convexity properties of $f$.
Focusing on
$E_{U_k}[\|e_k\|^2]$ and using the fact that 
$E[v_k] = 0,$ we get 
\begin{align}\label{eq:sceGD}
E_{U_k}[\|e_{k+1}\|^2] = E_{U_k}[\|e_k\|^2] + h^2E_{U_k}[\|v_k+l_k\|^2] - 2h E_{U_k}[\langle e_k , l_k \rangle].
\end{align}
Moreover, we have
\[E_{U_k}[\langle e_k , l_k \rangle] = E_{U_k}[\langle e_k , \nabla f_\mu(\tilde{z}_k)\rangle]\]
since \(E_{U_k}[\langle e_k , E_{U_k}[\nabla f_\mu(\tilde{z}_k)] \rangle] = 0.\) Now,
the strong convexity of $f$ and consequently strong convexity of $f_\mu$ (Lemma~\ref{lm:fmu_sc}) implies that
\[f_\mu(\bar{z}_k)\geq f_\mu(\tilde{z}_k) + \langle \nabla f_\mu(\tilde{z}_k), \bar{z}_k-\tilde{z}_k\rangle + \frac{\beta}{2}\| \bar{z}_k-\tilde{z}_k\|^2\]
or
\[E_{U_k}[\langle e_k , \nabla f_\mu(\tilde{z}_k)\rangle] \geq E_{U_k}[f_\mu(\tilde{z}_k)]-f_\mu(\bar{z}_k)  + \frac{\beta}{2}E_{U_k}[\|e_k\|^2].\]
Moreover, combining
strong convexity of $f_\mu$ and  
Jensen's inequality, leads to the estimates
\begin{align*}
    E_{U_k}[f_\mu(\tilde{z}_k)]\geq f_\mu(E_{U_k}[\tilde{z}_k])\geq f_\mu(\bar{z}_k), 
\end{align*}
and hence
\begin{align}
    E_{U_k}[\langle e_k , \nabla f_\mu(\tilde{z}_k)\rangle] \geq \frac{\beta}{2}E_{U_k}[\|e_k\|^2]. \label{eq:random_est_strong_convex}
\end{align}
Thus, substituting \eqref{eq:random_est_strong_convex} 
in \eqref{eq:sceGD}, we have
\begin{align}\label{eq:sceGD1}
E_{U_k}[\|e_{k+1}\|^2] \leq (1-h\beta)E_{U_k}[\|e_k\|^2] + h^2E_{U_k}[\|v_k+l_k\|^2].
\end{align}

We know that $\|v_k+l_k\|^2\leq2\|v_k\|^2+2\|l_k\|^2$ and $E[\|v_k\|^2]\leq\frac{\sigma^2}{t}.$ Similar to the estimates 
in \eqref{eq:l}, we can hence conclude the following chain of equalities and inequalities:
\begin{align}
\begin{split}\label{eq:l2}
    E[\|l_k\|^2] &= E_{U_k}[\|\nabla f_\mu(\tilde{z}_k)-\nabla f_\mu(\bar{z}_k)+\nabla f_\mu(\bar{z}_k)- E_{U_k}[\nabla f_\mu(\tilde{z}_k)]\|^2]\\&
    =E[\|\nabla f_\mu(\tilde{z}_k)-\nabla f_\mu(\bar{z}_k)\|^2+\|\nabla f_\mu(\bar{z}_k)- E_{U_k}[\nabla f_\mu(\tilde{z}_k)]\|^2]\\&\qquad
    -2E[\langle\nabla f_\mu(\tilde{z}_k)-\nabla f_\mu(\bar{z}_k),E_{U_k}[\nabla f_\mu(\tilde{z}_k)]-\nabla f_\mu(\bar{z}_k)\rangle]\\
    &\leq E[\|\nabla f_\mu(\tilde{z}_k)-\nabla f_\mu(\bar{z}_k)\|^2]\\&\leq L_1^2E[\|e_k\|^2].
\end{split}
\end{align}
Thus
\begin{align}\label{eq:sceGD2}
E_{U_k}[\|e_{k+1}\|^2] \leq \rho E_{U_k}[\|e_k\|^2] + \frac{2h^2\sigma^2}{t}
\end{align}
and where $\rho$ is defined as 
$\rho=1-h\beta+2L_1^2h^2$. To have $\rho<1,$ we require $h<\frac{\beta}{2L_1^2}.$ Since $e_0=0,$ by induction we get 
\begin{align}\label{eq:sceGD3}
E_{U_k}[\|e_{k}\|^2] \leq \frac{2h^2\sigma^2}{t(1-\rho)} = \frac{2h^2\sigma^2}{t(h\beta-2L_1^2h^2)} = \frac{2h\sigma^2}{t(\beta-2L_1^2h)}.
\end{align}
Hence, combining all the estimates above, we get
\begin{align*}
\sup_k \|q_k\| \leq h \frac{\mu}{2}L_1(n+3)^{3/2}+h L_1 \sqrt{\frac{2h\sigma^2}{t(\beta-2L_1^2h)}}
\end{align*}
which completes the proof.
\end{proof}
\begin{proof}[Proof of Theorem~\ref{th:GDsc}]
By Lemma~\ref{lem:q_GDsc}, the averaged ZO-GD dynamics~\eqref{eq:ZOGD_avg} take the form of the perturbed system~\eqref{eq:dt_perturbed} with perturbation satisfying the bound~\eqref{eq:q_bound_GDsc}. It remains to verify that the unperturbed GD dynamics satisfy the hypotheses of Theorem~\ref{thm:ISS_exp}.

For $\beta$-strongly convex $f$ with $L_1$-Lipschitz gradients and $h \in \bigl(0, \frac{2}{\beta + L_1}\bigr)$, from~\cite[Thm.~2.1.15]{nesterov2018lectures} we have
\[
\|z_{k+1} - z^*\|^2 \leq \Bigl(1 - \frac{2h\beta L_1}{\beta + L_1}\Bigr)\|z_k - z^*\|^2,
\]
where $z^* = \arg\min_{x\in \R^n} f(x)$. Consider the Lyapunov function $V(z) = \|z - z^*\|^2$. Then~\eqref{eq:sandwich} holds with $c_1 = c_2 = 1$, and the above contraction gives~\eqref{eq:decrease} with $c = \frac{2h\beta L_1}{\beta + L_1} \in (0,1)$. All hypotheses of Theorem~\ref{thm:ISS_exp} are therefore satisfied.

The ISS bound~\eqref{eq:ISS_bound} then guarantees that ZO-GD converges to a neighbourhood of $z^*$ with the same exponential rate $(1 - c)^{k/2}$ as GD. The radius of this neighbourhood is governed by $\gamma\bigl(\sup_k \|q_k\|\bigr)$, which, combined with the perturbation bound from Lemma~\ref{lem:q_GDsc}, yields~\eqref{eq:r_GDsc}. Since the bound~\eqref{eq:q_bound_GDsc} can be made smaller than $\bar{q}$ by choosing $\mu$ sufficiently small and $t$ sufficiently large, such parameter choices always exist.
\end{proof}

\begin{proof}[Proof of Lemma~\ref{lem:q_HBsc}]
Consider $q_k$ given in \eqref{eq:gen_q_unified} for HB, which satisfies
\begin{equation}\label{eq:qk_split}
\|q_k\| \;\leq\; h_1 \bigl\|
  E_{U_k}[\nabla f_\mu(x_k) - \nabla f(x_k)]
\bigr\|
+ h_1 \bigl\|
  E_{U_k}[\nabla f(x_k)] - \nabla f(\bar{x}_k)
\bigr\|.
\end{equation}
Using \cite[Lem.~3]{nesterov_random_2017}, we know that $\|\nabla f_\mu(x) - \nabla f(x)\| \leq \frac{\mu}{2}\, L_1\,(n+3)^{3/2}$ uniformly in $x$. Thus, the first term in the right-hand side \eqref{eq:qk_split} is bounded and can be ensured to be arbitrarily small by selecting
$h_1$ and $\mu$ appropriately. For the second term in the right-hand side of \eqref{eq:qk_split}, by $L_1$-Lipschitz continuity of $\nabla f$ and Jensen's inequality, we have
\begin{equation}\label{eq:nonlin_bound}
\bigl\| E_{U_k}[\nabla f(x_k)] - \nabla f(\bar{x}_k) \bigr\|
\;\leq\; L_1\, E_{U_k}[\|x_k - \bar{x}_k\|].
\end{equation}
Now let $e_k = z_k - \bar{z}_k$ so that $x_k - \bar{x}_k = T e_k$ where $T = [0 \;\; I_n]$. Thus we are interested in bounding $E[\|Te_k\|]$ or $E[\|Te_k\|^2]$. Subtracting the averaged dynamics \eqref{eq:avg_HBNAG_perturbed} (for HB) from the ZO-HB update gives
\begin{equation}\label{eq:error_dyn}
e_{k+1} = A\, e_k - h_1 
  \begin{bmatrix} 0 \\ v_k + l_k \end{bmatrix},
\end{equation}
where $v_k = g_\mu(x_k) - \nabla f_\mu(x_k)$ with $E[v_k] = 0$, and $l_k = \nabla f_\mu(x_k) - E_{U_k}[\nabla f_\mu(x_k)]$. As a next step, let
$x_k^e = Te_k$. From the lower block of \eqref{eq:error_dyn} we have
\begin{equation}\label{eq:xke_recur}
x_{k+1}^e = (1+h_2)\,x_k^e - h_2\,x_{k-1}^e - h_1(v_k + l_k).
\end{equation}
Let $a_k = E[\|x_k^e\|^2]$. Squaring \eqref{eq:xke_recur}, taking expectations, and using $E[\langle y, v_k\rangle] = 0$ for any $y$ independent of $u_k$, leads to the expression
\begin{align}\label{eq:ak_expand}
a_{k+1} 
&= E\bigl[\|(1+h_2)x_k^e - h_2 x_{k-1}^e\|^2\bigr]
 + h_1^2\,E[\|v_k+l_k\|^2]
 - 2h_1\,E\bigl[\langle (1+h_2)x_k^e - h_2 x_{k-1}^e,\; l_k\rangle\bigr].
\end{align}
We derive bounds for each term in \eqref{eq:ak_expand} separately.
For the deterministic quadratic term, expanding and applying Young's inequality provides the upper bound
\begin{align}\label{eq:det_bound}
\|(1+h_2)x_k^e - h_2 x_{k-1}^e\|^2
&= (1+h_2)^2\|x_k^e\|^2 + h_2^2\|x_{k-1}^e\|^2
  - 2h_2(1+h_2)\langle x_{k-1}^e, x_k^e\rangle \notag\\
&\leq (1+h_2)(1+2h_2)\,\|x_k^e\|^2
  + h_2(1+2h_2)\,\|x_{k-1}^e\|^2.
\end{align}
For the inner product, since 
$E[x_k^e] = 0$ and $E_{U_k}[\nabla f_\mu(x_k)]$ is deterministic, we get $E[\langle x_k^e, l_k\rangle] = E[\langle x_k - \bar{x}_k, \nabla f_\mu(x_k)\rangle]$. By $\beta$-strong convexity of $f_\mu$ (from Lemma~\ref{lm:fmu_sc}), the inequality
\[
\langle \nabla f_\mu(x_k),\, x_k - \bar{x}_k \rangle
\;\geq\; f_\mu(x_k) - f_\mu(\bar{x}_k) + \frac{\beta}{2}\,\|x_k - \bar{x}_k\|^2
\]
is thus satisfied.
Taking expectations and applying Jensen's inequality to
$E[f_\mu(x_k)] \geq f_\mu(\bar{x}_k)$ then yields
\begin{equation}\label{eq:sc_inner}
E[\langle x_k^e, l_k\rangle] \;\geq\; \tfrac{\beta}{2}\,a_k.
\end{equation}
For the cross term $E[\langle x_{k-1}^e, l_k\rangle]$, using $E[\|l_k\|] \leq 2L_1\,E[\|x_k^e\|]$ (from $L_1$-smoothness of $f_\mu$ and the triangle inequality) and Young's inequality, we know that
\begin{equation}\label{eq:mom_cross}
|E[\langle x_{k-1}^e, l_k\rangle]|
\;\leq\; L_1\bigl(a_{k-1} + a_k\bigr).
\end{equation}
Since we assume that the variance of the random oracle is bounded by 
$E[\|v_k\|^2] \leq \sigma^2/t$, for
$l_k$, similar to \eqref{eq:l2}, it holds that $E[\|l_k\|^2] \leq L_1^2\, a_k$ and thus
$E[\|v_k+l_k\|^2] \leq \frac{2\sigma^2}{t} + 2L_1^2\, a_k$.
 
Substituting all bounds into \eqref{eq:ak_expand} and letting 
\begin{align}
\phi &= (1+h_2)(1+2h_2) - h_1(1+h_2)\beta + 2h_1h_2 L_1 + 2h_1^2 L_1^2, \\
c_0 &= h_2(1+2h_2) + 2h_1h_2 L_1,
\end{align}
we get the two-step recursion
\begin{equation}\label{eq:two_step}
a_{k+1} \;\leq\; \phi\, a_k + c_0\, a_{k-1} + \frac{2h_1^2\sigma^2}{t}.
\end{equation}
As a next step, we define
$m_k = \max\{a_k, a_{k-1}\}$. Since $a_k \leq m_k$ and $a_{k-1} \leq m_k$, \eqref{eq:two_step} gives $a_{k+1} \leq (\phi + c_0)\,m_k + \frac{2h_1^2\sigma^2}{t}$. Since additionally 
$a_k \leq m_k$ holds, $m_{k+1}$ can be upper bounded by
\begin{equation}\label{eq:mk_recur2}
m_{k+1} = \max\{a_{k+1}, a_k\} \;\leq\; \max\left\{(\phi+c_0)\, m_k + \frac{2h_1^2\sigma^2}{t},m_k \right\},
\end{equation}
From 
$m_0 = 0$ (as $e_0 = 0$), we get 
\begin{equation}\label{eq:mk_recur}
m_{k+1} = \max\{a_{k+1}, a_k\} \;\leq\; (\phi+c_0)\, m_k + \frac{2h_1^2\sigma^2}{t},
\end{equation}
which is a standard one-step contraction provided $\rho_{\mathrm{HB}} := \phi + c_0 < 1$. Computing $\rho_{\mathrm{HB}}$ explicitly, we get
\begin{equation}\label{eq:rho_HB}
\rho_{\mathrm{HB}} = \phi + c_0 = (1+2h_2)^2 - h_1(1+h_2)\beta + 4h_1h_2 L_1 + 2h_1^2 L_1^2.
\end{equation}
Thus $\rho_{\mathrm{HB}} < 1$ requires
\begin{equation}\label{eq:param_cond}
h_1(1+h_2)\beta > 4h_2 + 4h_2^2 + 4h_1h_2 L_1 + 2h_1^2 L_1^2
\end{equation}
to hold. To find a sufficient condition for \eqref{eq:param_cond} to be satisfied let $h_2\in(0,ch_1)$. Then it is sufficient to have 
\[
h_1\beta>2h_1^2L_1^2+4ch_1+4c^2h_1^2+4ch_1^2L_1 \quad \text{or equivalently} \quad \beta>2h_1L_1^2+4c+4c^2h_1+4ch_1L_1.
\]
Thus we need $c\in(0,\frac{\beta}{4})$ and $h_1\in(0,\frac{\beta-4c}{4c^2+4cL_1+2L_1^2})$ to guarantee that \eqref{eq:param_cond} is satisfied, providing the bounds stated in Lemma \ref{lem:q_HBsc}.

Since $m_0 = 0$ (as $e_0 = 0$), iterating \eqref{eq:mk_recur} leads to
\[
\sup_k\, a_k \;\leq\; \sup_k\, m_k \;\leq\; \frac{2h_1^2\sigma^2}{t(1-\rho_{\mathrm{HB}})}
\]
and therefore, by Jensen's inequality we have
\begin{equation}\label{eq:Te_uniform}
\sup_k\,E[\|x_k - \bar{x}_k\|]
\;\leq\; \sup_k\sqrt{a_k}
\;\leq\; \frac{h_1\sigma}{\sqrt{t}}\,
  \sqrt{\frac{2}{1-\rho_{\mathrm{HB}}}}.
\end{equation}
Combining \eqref{eq:qk_split}, the smoothing bias bound, \eqref{eq:nonlin_bound}, and \eqref{eq:Te_uniform} we can thus conclude that
\begin{align}
\sup_k\,\|q_k\|
&\leq \frac{h_1\mu\,L_1(n+3)^{3/2}}{2}
  + h_1 L_1 \cdot \frac{h_1\sigma}{\sqrt{t}}\,
    \sqrt{\frac{2}{1-\rho_{\mathrm{HB}}}} \notag\\[4pt]
&= \frac{h_1\mu\,L_1(n+3)^{3/2}}{2}
  + \frac{h_1^2 L_1\,\sigma}{\sqrt{t}}\,
    \sqrt{\frac{2}{1-\rho_{\mathrm{HB}}}}\,,
\end{align}
which completes the proof.
\end{proof}
\begin{proof}[Proof of Theorem~\ref{th:HBsc}]
By Lemma~\ref{lem:q_HBsc}, the averaged ZO-HB dynamics~\eqref{eq:avg_HBNAG_perturbed} take the form of the perturbed system~\eqref{eq:dt_perturbed} with bounded perturbation. For $\beta$-strongly convex $f$ with $L_1$-Lipschitz gradients and parameters satisfying the hypotheses of Lemma~\ref{lem:q_HBsc}, HB locally converges to the minimiser $z^*$ exponentially fast~\cite{polyak1964some}. By the converse Lyapunov theorem for exponentially stable discrete-time systems (see, e.g.,~\cite[Thm.~5.6]{kellett2023introduction}), there exists a Lyapunov function $V$ satisfying~\eqref{eq:sandwich}--\eqref{eq:decrease} on some ball $\mathcal{B}_R(z^e)$, $R>0$, in the augmented state space $\mathbb{R}^{2n}$. All hypotheses of Theorem~\ref{thm:ISS_exp} are therefore satisfied, and the conclusion follows by the same argument as in the proof of Theorem~\ref{th:GDsc}.
\end{proof}

\begin{proof}[Proof of Lemma~\ref{lem:q_NAGsc}]
Consider $q_k$ given in \eqref{eq:gen_q_unified} for NAG, which satisfies
\begin{equation}\label{eq:qk_split_NAG}
\|q_k\| \leq h_1 \bigl\| E_{U_k}[\nabla f_\mu(TAz_k) - \nabla f(TAz_k)] \bigr\| + h_1 \bigl\| E_{U_k}[\nabla f(TAz_k)] - \nabla f(TA\bar{z}_k) \bigr\|.
\end{equation}
Using \cite[Lem.~3]{nesterov_random_2017}, we know that $\|\nabla f_\mu(x) - \nabla f(x)\| \leq \frac{\mu}{2}L_1(n+3)^{3/2}$ uniformly in $x$. Thus, the first term in the right-hand side \eqref{eq:qk_split_NAG} is bounded and can be controlled by the choice of $h_1$ and $\mu$. For the second term in the right-hand side of \eqref{eq:qk_split_NAG}, by $L_1$-Lipschitz continuity of $\nabla f$ and Jensen's inequality, we have
\begin{equation}\label{eq:nonlin_bound_NAG}
\bigl\| E_{U_k}[\nabla f(TAz_k)] - \nabla f(TA\bar{z}_k) \bigr\| \leq L_1\, E_{U_k}[\|TAz_k - TA\bar{z}_k\|] = L_1\, E_{U_k}[\|TAe_k\|],
\end{equation}
where $e_k = z_k - \bar{z}_k$. Thus, we need to derive an upper bound on 
$E[\|TAe_k\|]$ or $E[\|TAe_k\|^2]$. Let $x_k^e = Te_k$ and define the lookahead error $y_k^e = TAe_k = (1+h_2)x_k^e - h_2 x_{k-1}^e$. Subtracting the averaged dynamics \eqref{eq:avg_HBNAG_perturbed} (for NAG) from the ZO-NAG update gives
\begin{equation}\label{eq:error_dyn_NAG}
e_{k+1} = A\,e_k - h_1 \begin{bmatrix} 0 \\ v_k + l_k \end{bmatrix},
\end{equation}
where $v_k = g_\mu(TAz_k) - \nabla f_\mu(TAz_k)$ with $E[v_k] = 0$, and $l_k = \nabla f_\mu(TAz_k) - E_{U_k}[\nabla f_\mu(TAz_k)]$. From the lower block of \eqref{eq:error_dyn_NAG} we obtain
\begin{equation}\label{eq:xke_recur_NAG}
x_{k+1}^e = y_k^e - h_1(v_k + l_k).
\end{equation}
Let $a_k = E[\|x_k^e\|^2]$ and $b_k = E[\|y_k^e\|^2]$. Squaring \eqref{eq:xke_recur_NAG}, taking expectations, and using the fact that $E[\langle y_k^e, v_k\rangle] = 0$ leads to the equation
\begin{equation}\label{eq:ak_expand_NAG}
a_{k+1} = b_k + h_1^2\,E[\|v_k + l_k\|^2] - 2h_1\,E[\langle y_k^e,\, l_k\rangle].
\end{equation}
To proceed, we derive upper bounds for each term in the right-hand side.
For the inner product, since $E[y_k^e] = 0$ and $E_{U_k}[\nabla f_\mu(TAz_k)]$ is deterministic, we get $E[\langle y_k^e, l_k\rangle] = E[\langle TAe_k, \nabla f_\mu(TAz_k)\rangle]$. By $\beta$-strong convexity of $f_\mu$ (Lemma~\ref{lm:fmu_sc}), it additionally holds that
\[
\langle \nabla f_\mu(TAz_k),\, TAz_k - TA\bar{z}_k\rangle \geq f_\mu(TAz_k) - f_\mu(TA\bar{z}_k) + \frac{\beta}{2}\|TAe_k\|^2.
\]
Taking expectations and applying Jensen's inequality to
$E[f_\mu(TAz_k)] \geq f_\mu(E[TAz_k]) = f_\mu(TA\bar{z}_k)$ allows us to write
\begin{equation}\label{eq:sc_inner_NAG}
E[\langle y_k^e, l_k\rangle] \geq \frac{\beta}{2}\,b_k.
\end{equation}
This is the key advantage of the NAG structure compared to the HB structure. The strong convexity contraction acts directly on $b_k = E[\|y_k^e\|^2]$, which is the same quantity that appears with a positive sign in \eqref{eq:ak_expand_NAG}. There is no separate momentum cross term to bound.
 
For the noise, since $l_k = \nabla f_\mu(TAz_k) - E_{U_k}[\nabla f_\mu(TAz_k)]$, using $L_1$-smoothness and similar to the process in \eqref{eq:l} and \eqref{eq:l2}, $E[\|l_k\|] \leq 2L_1\,E[\|y_k^e\|]$, so $E[\|l_k\|^2] \leq L_1^2\,b_k$ and with 
$E[\|v_k\|^2] \leq \sigma^2/t$, we get $E[\|v_k + l_k\|^2] \leq \frac{2\sigma^2}{t} + 2L_1^2\,b_k$.
Substituting these estimates into \eqref{eq:ak_expand_NAG} provides the upper bound
\begin{equation}\label{eq:ak_bk_NAG}
a_{k+1} \leq (1 - h_1\beta + 2h_1^2 L_1^2)\,b_k + \frac{2h_1^2\sigma^2}{t}.
\end{equation}
Let $\psi = 1 - h_1\beta + 2h_1^2 L_1^2$. For the quadratic term $b_k = E[\|y_k^e\|^2]$, expanding and applying Young's inequality exactly as in \eqref{eq:det_bound} yields
\begin{equation}\label{eq:bk_bound_NAG}
b_k \leq (1+h_2)(1+2h_2)\,a_k + h_2(1+2h_2)\,a_{k-1}.
\end{equation}
Substituting \eqref{eq:bk_bound_NAG} into \eqref{eq:ak_bk_NAG} and letting
\begin{align*}
\phi &= \psi\,(1+h_2)(1+2h_2), \\
c_0 &= \psi\,h_2(1+2h_2),
\end{align*}
we get the two-step recursion
\begin{equation}\label{eq:two_step_NAG}
a_{k+1} \leq \phi\,a_k + c_0\,a_{k-1} + \frac{2h_1^2\sigma^2}{t}.
\end{equation}
Next, we define $m_k = \max\{a_k, a_{k-1}\}$. Since $a_k \leq m_k$ and $a_{k-1} \leq m_k$, \eqref{eq:two_step_NAG} gives $a_{k+1} \leq (\phi + c_0)\,m_k + \frac{2h_1^2\sigma^2}{t}$ and from 
$a_k \leq m_k$ we have
\begin{equation}\label{eq:mk_recur_NAG2}
m_{k+1} = \max\{a_{k+1}, a_k\} \leq \max\bigl\{(\phi+c_0)\,m_k + \tfrac{2h_1^2\sigma^2}{t},\; m_k\bigr\}.
\end{equation}
Since $m_0 = 0$ (as $e_0 = 0$), we get
\begin{equation}\label{eq:mk_recur_NAG}
m_{k+1} \leq (\phi + c_0)\,m_k + \frac{2h_1^2\sigma^2}{t},
\end{equation}
which is a standard one-step contraction provided $\rho_{\mathrm{NAG}} := \phi + c_0 < 1$ and computing $\rho_{\mathrm{NAG}}$ 
explicitly yields
\begin{equation}\label{eq:rho_NAG}
\rho_{\mathrm{NAG}} = \psi\,(1+2h_2)^2 = (1 - h_1\beta + 2h_1^2 L_1^2)(1+2h_2)^2.
\end{equation}
Thus $\rho_{\mathrm{NAG}} < 1$ requires $\psi(1+2h_2)^2 < 1$, i.e.,
\begin{equation}\label{eq:param_cond_NAG}
(1+2h_2)^2 < \frac{1}{1-h_1\beta + 2h_1^2 L_1^2}.
\end{equation}
Choosing $h_1 \in (0, \frac{\beta}{2L_1^2})$ ensures $\psi \in (0,1)$, and then any $h_2 \in (0, \frac{1}{4}(\psi^{-1/2} - 1))$ guarantees \eqref{eq:param_cond_NAG}.
 
Since $m_0 = 0$, iterating \eqref{eq:mk_recur_NAG}, we then have
\[
\sup_k\, a_k \leq \sup_k\, m_k \leq \frac{2h_1^2\sigma^2}{t(1-\rho_{\mathrm{NAG}})}.
\]
To bound $E[\|TAe_k\|]$, note that $\|TAe_k\|^2 = \|y_k^e\|^2 = b_k$, and from \eqref{eq:bk_bound_NAG} we obtain
\[
\sup_k\, E[\|y_k^e\|^2] \leq (1+2h_2)^2 \sup_k m_k \leq \frac{2(1+2h_2)^2h_1^2\sigma^2}{t(1-\rho_{\mathrm{NAG}})}.
\]
Finally, by
Jensen's inequality,
\begin{equation}\label{eq:TAe_uniform}
\sup_k\, E[\|TAe_k\|] \leq \frac{(1+2h_2)h_1\sigma}{\sqrt{t}}\sqrt{\frac{2}{1-\rho_{\mathrm{NAG}}}},
\end{equation}
and combining 
\eqref{eq:qk_split_NAG}, the smoothing bias bound, \eqref{eq:nonlin_bound_NAG}, and \eqref{eq:TAe_uniform} it holds that
\begin{align}
\sup_k\,\|q_k\|
&\leq \frac{h_1\mu\,L_1(n+3)^{3/2}}{2} + h_1 L_1 \cdot \frac{(1+2h_2)h_1\sigma}{\sqrt{t}}\sqrt{\frac{2}{1-\rho_{\mathrm{NAG}}}} \notag\\[4pt]
&= \frac{h_1\mu\,L_1(n+3)^{3/2}}{2} + \frac{(1+2h_2)h_1^2 L_1\sigma}{\sqrt{t}}\sqrt{\frac{2}{1-\rho_{\mathrm{NAG}}}}\,,
\end{align}
which completes the proof.
\end{proof}

\begin{proof}[Proof of Theorem~\ref{th:NAGsc}]
By Lemma~\ref{lem:q_NAGsc}, the averaged ZO-NAG dynamics~\eqref{eq:avg_HBNAG_perturbed} take the form of the perturbed system~\eqref{eq:dt_perturbed} with bounded perturbation. For $\beta$-strongly convex $f$ with $L_1$-Lipschitz gradients and parameters satisfying the hypotheses of Lemma~\ref{lem:q_NAGsc}, NAG converges to the minimiser $z^e$
exponentially fast~\cite{nesterov2018lectures}. By the converse Lyapunov theorem for exponentially stable discrete-time systems (see, e.g.,~\cite[Thm.~5.6]{kellett2023introduction}), there exists a Lyapunov function $V$ satisfying~\eqref{eq:sandwich}, \eqref{eq:decrease} on some ball $\mathcal{B}_R(z^e)$, $R>0$, in the augmented state space $\mathbb{R}^{2n}$. All hypotheses of Theorem~\ref{thm:ISS_exp} are therefore satisfied, and the conclusion follows by the same argument as in the proof of Theorem~\ref{th:GDsc}.
\end{proof}
\begin{proof}[Proof of Lemma~\ref{lem:q_GDreg}]
     Consider $q_k$ given in \eqref{eq:gen_q_ZOGD_reg}, which satisfies
    \begin{align}\label{eq:q_ZOGD_norm_reg}
        \|q_k\|\leq h\|\nabla f(\bar{z}_k)- \nabla f_\mu(\bar{z}_k)\|+h\|\nabla f_\mu(\bar{z}_k) - E_{U_k}[\nabla f_\mu(\tilde{z}_k)]\|
    \end{align}
    Using \cite[Lem. 3]{nesterov_random_2017}, we know that \( \|\nabla f(\bar{z}_k)- \nabla f_\mu(\bar{z}_k)\|\leq \tfrac{\mu}{2}L_1(n+3)^{3/2}.\) Thus, the first term in the right-hand side of \eqref{eq:q_ZOGD_norm_reg} is bounded and can be manipulated by choosing
    $h$ and $\mu$ appropriately. For the second term in the right-hand side of \eqref{eq:q_ZOGD_norm_reg}, we have 
    \begin{align}
        \|E_{U_k}[\nabla f_\mu(\tilde{z}_k)]- \nabla f_\mu(\bar{z}_k)\| &\leq \|E_{U_k}[\nabla f_\mu(\tilde{z}_k)- \nabla f_\mu(\bar{z}_k)]\|\notag\\
        &\leq E_{U_k}[\|\nabla f_\mu(\tilde{z}_k)- \nabla f_\mu(\bar{z}_k)\|]\notag\\
        &\leq L_1E_{U_k}[\|\tilde{z}_k- \bar{z}_k\|]
    \end{align}
    Now, let $e_k = \tilde{z}_k- \bar{z}_k.$ As a next step, we derive an upper bound on
    $E[\|e_k\|^2].$ We know that 
    \begin{align*}
        e_{k+1} &= ce_k -h g_\mu(\tilde{z}_k) +h E_{U_k}[\nabla f_\mu(\tilde{z}_k)] \\&=e_k - h (g_\mu(\tilde{z}_k)- \nabla f_\mu(\tilde{z}_k)) -h (\nabla f_\mu(\tilde{z}_k) - E_{U_k}[\nabla f_\mu(\tilde{z}_k)])
    \end{align*}
    To proceed, we use the notation
    $v_k = g_\mu(\tilde{z}_k)- \nabla f_\mu(\tilde{z}_k)$ and $l_k =  \nabla f_\mu(\tilde{z}_k) - E_{U_k}[\nabla f_\mu(\tilde{z}_k)]$ and thus 
    $e_{k+1} = ce_k -h (v_k+l_k).$ We assume the variance of the random oracle is bounded and get $E[\|v_k\|^2]\leq\frac{\sigma^2}{t}.$ Moreover, for $l_k$
    the following chain of inequalities is satisfied:
    \begin{align}\label{eq:l_reg}
        E[\|l_k\|] \leq E[\|\nabla f_\mu(\tilde{z}_k)-\nabla f_\mu(\bar{z}_k)\|] + E[\|\nabla f_\mu(\bar{z}_k)-\nabla f_\mu(\tilde{z}_k)\|]\leq2L_1E[\|e_k\|].
    \end{align}
From the fact that
$E[v_k] = 0,$ we 
get the following bound on $E_{U_k}[\|e_{k+1}\|^2]$,
\begin{align}\label{eq:GD_reg}
E_{U_k}[\|e_{k+1}\|^2] &= c^2E_{U_k}[\|e_k\|^2] + h^2E_{U_k}[\|v_k+l_k\|^2] - 2ch E_{U_k}[\langle e_k , l_k \rangle]\\
&\leq c^2E_{U_k}[\|e_k\|^2] + h^2E_{U_k}[\|v_k+l_k\|^2] + 2ch E_{U_k}[\|e_k\|\|l_k\|]\\
&\leq (c^2+4cL_1h)E_{U_k}[\|e_k\|^2] + h^2E_{U_k}[\|v_k+l_k\|^2],
\end{align}
where the last inequality uses Jensen's inequality and steps similar to same as
\eqref{eq:l_reg}.
Moreover, we know that $\|v_k+l_k\|^2\leq2\|v_k\|^2+2\|l_k\|^2$ and $E[\|v_k\|^2]\leq\frac{\sigma^2}{t}.$ Similar to the process in \eqref{eq:l_reg} and considering the fact that $E[\|X-E[X]\|^2]\leq E[\|X\|^2]$, we have
$E[\|l_k\|^2] \leq L_1^2E[\|e_k\|^2]$ and thus
\begin{align}\label{eq:GD_reg2}
E_{U_k}[\|e_{k+1}\|^2] \leq \rho E_{U_k}[\|e_k\|^2] + \frac{2h^2\sigma^2}{t},
\end{align}
where $\rho=c^2+4cL_1h+2L_1^2h^2$. For 
$c\in[0,1)$, to 
have $\rho<1,$ we require $h\in\bigl(0,\frac{\sqrt{2(1+c^2)}-2c}{2L_1}\bigr).$ Since $e_0=0,$ by induction we get 
\begin{align}\label{eq:GD_reg3}
E_{U_k}[\|e_{k}\|^2] \leq \frac{2h^2\sigma^2}{t(1-\rho)} = \frac{2h^2\sigma^2}{t(1-c^2-4cL_1h-2L_1^2h^2)}.
\end{align}
Hence, combining the derivations above, we get
\begin{align}
\sup_k \|q_k\| \leq h \frac{\mu}{2}L_1(n+3)^{3/2}+h L_1 \sqrt{\frac{2h^2\sigma^2}{t(1-c^2-4cL_1h-2L_1^2h^2)}}
\end{align} 
which completes the proof.
\end{proof}

\begin{proof}[Proof of Theorem~\ref{th:GD_reg}]
By Lemma~\ref{lem:q_GDreg}, the averaged ZO-GD dynamics~\eqref{eq:ZOGD_avg_reg} take the form of the perturbed system~\eqref{eq:dt_perturbed} with perturbation satisfying the bound~\eqref{eq:q_bound_GD_reg}. It remains to verify that the unperturbed regularised GD dynamics~\eqref{eq:GD_alg_reg} satisfy the hypotheses of Theorem~\ref{thm:ISS_asymp}.

By assumption, GD applied to the regularised problem converges to a fixed point $z^e$. Since we have convergence for GD, the converse Lyapunov theorem for asymptotically stable discrete-time systems (see, e.g.,~\cite[Thm.~5.5 and Thm.~5.6]{kellett2023introduction})
guarantees the existence of a Lyapunov function $V$ satisfying~\eqref{eq:sandwich1}--\eqref{eq:decrease1} on some ball $\mathcal{B}_R(z^e)$, $R>0$. All hypotheses of Theorem~\ref{thm:ISS_asymp} are therefore satisfied.

The ISS bound~\eqref{eq:ISS} then guarantees that ZO-GD converges to a neighbourhood of $z^e$ with the same transient decay profile as GD. The radius of this neighbourhood is governed by $\gamma\bigl(\sup_k \|q_k\|\bigr)$, which, combined with the perturbation bound from Lemma~\ref{lem:q_GDreg}, yields~\eqref{eq:r_GD_reg}. Since the bound~\eqref{eq:q_bound_GD_reg} can be made smaller than $\bar{q}$ by choosing $\mu$ sufficiently small and $t$ sufficiently large, such parameter choices always exist.
\end{proof}

\section{Complementary remarks on the results of Section~\ref{sec:main}}\label{app:remarks}
The following remarks provide additional technical details on the results presented in Section~\ref{sec:main}.
\begin{remark}\label{rem:dim_h}
    While the first term in~\eqref{eq:r_GDsc} can be manipulated 
    by choosing $\mu$ small, the second term is governed by $h$ and $t$. In the regime of large variance $\sigma$, for ZO-GD to operate with the largest admissible step size of GD, one must increase $t$, resulting in more function evaluations per iteration. When the optimal step size of GD is of order $\sigma^{-2/3}$, a small value of $t$ suffices for ZO-GD to track GD closely with the same computational cost. Moreover, Theorem~\ref{th:GDsc} requires $h\in\bigl(0,\frac{\beta}{2L_1^2}\bigr)$, whereas for strongly convex functions GD requires $h\in\bigl(0,\frac{2}{\beta+L_1}\bigr)$~\cite{nesterov2018lectures}. Thus, if $\beta\sim L_1$, the two step-size bounds are of the same order.
\end{remark}

\begin{remark}
We note that the expected performance of ZO-GD for strongly convex objective functions in minimisation problems has been analysed in \cite[Thm.~8]{nesterov_random_2017}. Compared to Theorem~\ref{th:GDsc}, the result in \cite[Thm.~8]{nesterov_random_2017} exhibits an explicit dependence of the iteration complexity on the dimension, and does not account for the variance bound $\sigma$ or the number of sampled directions $t$. Moreover, the prescribed step size scales inversely with $n$, which is not consistent with standard GD.
In contrast, Theorem~\ref{th:GDsc} shows that ZO-GD, in expectation, achieves the same decay rate as GD and converges to a neighbourhood of the GD fixed point, whose
size depends on $\sigma$ and can be manipulated 
via $h$, $\mu$, and $t$. Furthermore, for any step size consistent with GD, one can select $\mu$ and $t$ such that the perturbation is arbitrarily small
and closely track GD.
\end{remark}

\begin{remark}[Relaxation of Assumption~1]\label{rem:relaxed_variance}
Assumption~1 requires the variance of the oracle $g_\mu$ defined in~\eqref{eq:eq23} to be uniformly bounded by a constant~$\sigma^2$.
In some scenarios, the variance of the zeroth-order oracle 
scales with the magnitude of the gradient at the query point.
Here, we show 
analogous results to Lemma~2 and Theorem~3 under the following relaxed assumption. 
Analogous results for 
other lemmas and theorems can be derived following similar steps.

\begin{assumption}
The variance of the oracle $g_\mu$ defined in~\eqref{eq:eq23} satisfies
\begin{equation}\label{eq:relaxed_var}
\mathbb{E}_u\!\left[\left\|g_\mu(x) - \nabla f_\mu(x)\right\|^2\right] \leq \frac{\sigma_0^2}{t} + \frac{\sigma_1^2}{t}\left\|\nabla f(x)\right\|^2,
\end{equation}
where $\sigma_0, \sigma_1 \geq 0$.
\end{assumption}

Note that Assumption~1 is recovered by setting $\sigma_1 = 0$ and $\sigma_0 = \sigma$.
Under this relaxed assumption, the proof of Lemma~2 proceeds as follows.
Consider the recursion for $\mathbb{E}_{U_k}[\|e_{k+1}\|^2]$ derived in~\eqref{eq:sceGD1}.
With $\mathbb{E}[\|v_k\|^2] \leq \sigma^2/t$, equation~(55) in the proof of Lemma~2 gives
\[
\mathbb{E}_{U_k}\!\left[\|e_{k+1}\|^2\right] \leq \rho\, \mathbb{E}_{U_k}\!\left[\|e_k\|^2\right] + \frac{2h^2\sigma^2}{t},
\]
where $\rho = 1 - h\beta + 2L_1^2 h^2$.
Under~\eqref{eq:relaxed_var}, the additive term becomes state-dependent.
Using the $L_1$-Lipschitz continuity of $\nabla f$ and $\nabla f(z^*) = 0$, we have
\[
\left\|\nabla f(\tilde{z}_k)\right\|^2 = \left\|\nabla f(\tilde{z}_k) - \nabla f(z^*)\right\|^2 \leq L_1^2 \left\|\tilde{z}_k - z^*\right\|^2 \leq 2L_1^2 \left\|e_k\right\|^2 + 2L_1^2 \left\|\bar{z}_k - z^*\right\|^2,
\]
where the last inequality uses $\|\tilde{z}_k - z^*\|^2 \leq 2\|e_k\|^2 + 2\|\bar{z}_k - z^*\|^2$.
Defining $\Sigma_1 := \sigma_1^2 L_1^2$, in place of~(55) we obtain
\begin{equation}\label{eq:relaxed_ek}
\mathbb{E}_{U_k}\!\left[\|e_{k+1}\|^2\right] \leq \left(\rho + \frac{4h^2\Sigma_1}{t}\right) \mathbb{E}\!\left[\|e_k\|^2\right] + \frac{2h^2\sigma_0^2}{t} + \frac{4h^2\Sigma_1}{t}\left\|\bar{z}_k - z^*\right\|^2.
\end{equation}
Unlike the case of Assumption~1, bounding $\mathbb{E}[\|e_k\|^2]$ alone is no longer sufficient, since~\eqref{eq:relaxed_ek} couples the error $e_k = \tilde{z}_k - \bar{z}_k$ to the averaged trajectory $\bar{z}_k$.
We therefore analyse the joint evolution of $a_k := \mathbb{E}[\|e_k\|^2]$ and $b_k := \|\bar{z}_k - z^*\|^2$.
Equation~\eqref{eq:relaxed_ek} gives
\begin{equation}\label{eq:ak_recursion}
a_{k+1} \leq \rho'\, a_k + \frac{2h^2\sigma_0^2}{t} + \frac{4h^2\Sigma_1}{t}\, b_k,
\end{equation}
where $\rho' := \rho + 4h^2\Sigma_1/t$.
For the $b_k$ recursion, recall the averaged dynamics~(16): $\bar{z}_{k+1} = \bar{z}_k - h\nabla f(\bar{z}_k) + q_k$.
Expanding, we have
\begin{align*}
b_{k+1} &= \left\|\bar{z}_k - h\nabla f(\bar{z}_k) + q_k - z^*\right\|^2 \\
&= \left\|\bar{z}_k - z^* - h\nabla f(\bar{z}_k)\right\|^2 + \|q_k\|^2 + 2\left\langle \bar{z}_k - z^* - h\nabla f(\bar{z}_k),\, q_k \right\rangle \\
&\leq \left(\sqrt{(1-c)\,b_k} + \|q_k\|\right)^2,
\end{align*}
where $c = 2h\beta L_1/(\beta + L_1) \in (0,1)$ is the GD contraction constant derived 
in \cite[Thm 2.1.15]{nesterov2018lectures}.
Using Young's inequality, for any $\varepsilon > 0$,
\begin{equation}\label{eq:bk_young}
b_{k+1} \leq (1+\varepsilon)(1-c)\, b_k + \left(1 + \tfrac{1}{\varepsilon}\right)\|q_k\|^2.
\end{equation}
From \eqref{eq:q_ZOGD_norm} and Jensen's inequality, $\|q_k\|^2 \leq 2\,\mathrm{bias}^2 + 2h^2 L_1^2\, a_k$, where $\mathrm{bias} := h\mu L_1(n+3)^{3/2}/2$ denotes the smoothing bias term.
Substituting into~\eqref{eq:bk_young} yields
\begin{equation}\label{eq:bk_recursion}
b_{k+1} \leq (1+\varepsilon)(1-c)\, b_k + 2\!\left(1+\tfrac{1}{\varepsilon}\right) h^2 L_1^2\, a_k + 2\!\left(1+\tfrac{1}{\varepsilon}\right) \mathrm{bias}^2.
\end{equation}
Combining~\eqref{eq:ak_recursion} and~\eqref{eq:bk_recursion} and considering their worst-case setting, i.e., equality, the coupled system can be written in matrix form as
\begin{equation}\label{eq:coupled_system}
\begin{bmatrix} a_{k+1} \\ b_{k+1} \end{bmatrix}
= M
\begin{bmatrix} a_k \\ b_k \end{bmatrix}
+ \begin{bmatrix} 2h^2\sigma_0^2/t \\ 2(1+1/\varepsilon)\,\mathrm{bias}^2 \end{bmatrix},
\end{equation}
where
\begin{equation}\label{eq:M_matrix}
M = \begin{bmatrix}
\rho' & \dfrac{4h^2\Sigma_1}{t} \\[8pt]
2\!\left(1+\dfrac{1}{\varepsilon}\right) h^2 L_1^2 & (1+\varepsilon)(1-c)
\end{bmatrix}.
\end{equation}
If spectral radius of $M$ is less than $1$, the coupled system~\eqref{eq:coupled_system} converges to a neighbourhood of the origin determined by $(I - M)^{-1}$ applied to the constant vector, and both $\sup_k a_k$ and $\sup_k b_k$ are uniformly bounded, which means that $sup_k \|q_k\|$ is uniformly bounded.
A sufficient condition for spectral radius of $M$ being less than $1$ is that both diagonal entries are strictly less than~$1$ and
\begin{equation}\label{eq:spectral_condition}
(1 - \rho')\bigl(1 - (1+\varepsilon)(1-c)\bigr) > \frac{4h^2\Sigma_1}{t} \cdot 2\!\left(1+\tfrac{1}{\varepsilon}\right) h^2 L_1^2,
\end{equation}
This follows from standard bounds on the spectral radius of a nonnegative matrix
\cite[Ch~1,~5, and~8]{horn2012matrix}. The condition $\rho' < 1$ is satisfiable for any $\Sigma_1 > 0$ by choosing $t$ sufficiently large or $h$ sufficiently small.
The condition $(1+\varepsilon)(1-c) < 1$ is satisfied for $\varepsilon< \frac{1}{1-c}-1$.
The cross condition~\eqref{eq:spectral_condition} has a right-hand side of order $O(h^4\Sigma_1 L_1^2/t)$ versus a left-hand side of order $O(h^2)$, which is satisfiable for $h$ small enough or $t$ large enough.

Once uniform boundedness of $a_k$ and $b_k$ is established, the perturbation bound takes the form
\[
\sup_k \|q_k\| \leq \mathrm{bias} + hL_1 \sup_k \sqrt{a_k},
\]
which is finite and controllable through the parameters $h$, $\mu$, and $t$.
The remainder of the ISS argument (Theorem~\ref{th:GDsc}) then applies without modification, yielding the same exponential decay rate and convergence to a neighbourhood whose radius depends on the perturbation bound.
The key difference compared to Assumption~1 is that $t$ must now be chosen large enough or $h$ be chosen small enough to ensure that the spectral radius of $M$ is less than $1$, introducing a coupling between the variance parameter~$\Sigma_1 = \sigma_1^2 L_1^2$ and the GD contraction rate~$c$.
\end{remark}

\begin{remark}
Setting $h_2 = 0$ reduces the heavy ball method to gradient descent. In this case, $\rho_{\mathrm{HB}}$ reduces to $\rho_{\mathrm{GD}} = 1 - h_1\beta + 2h_1^2 L_1^2$, and the bound recovers the result of the GD case in Section~\ref{sec:GDsc}.
\end{remark}
 
\begin{remark}
The bound~\eqref{eq:q_bound_HBsc} consists of two terms: the first term, $\frac{h_1\mu\,L_1(n+3)^{3/2}}{2}$, is the smoothing bias and which depends on
the smoothing parameter $\mu$; the second term captures the variance of the ZO oracle amplified through the HB dynamics and can be manipulated through
the number of samples $t$, the step size $h_1$, and the momentum parameter $h_2$ (through $\rho_{\mathrm{HB}}$).
\end{remark}

\begin{remark}
Setting $h_2 = 0$ reduces NAG to gradient descent. In this case, $\rho_{\mathrm{NAG}}$ reduces to $\rho_{\mathrm{GD}} = 1 - h_1\beta + 2h_1^2 L_1^2$, and the bound recovers the result of the GD case in Section~\ref{sec:GDsc}.
\end{remark}
 
\begin{remark}
The bound~\eqref{eq:q_bound_NAGsc} has the same two-term structure as the HB bound~\eqref{eq:q_bound_HBsc}: the first term is the smoothing bias (depending on
$\mu$), and the second term captures the ZO oracle variance amplified through the dynamics (and depending on
$t$, $h_1$, and $h_2$ through $\rho_{\mathrm{NAG}}$).
\end{remark}

\begin{remark}\label{rem:reg_sc}
We note that the step-size condition in Lemma~\ref{lem:q_GDreg}, combined 
with the parametrisation $\lambda = \frac{1-c}{2h}$, implies 
$\lambda > \frac{L_1}{2}$, so the regularised objective is 
strongly convex across the entire admissible parameter range. 
Nevertheless, the $L_2$ regularisation framework remains 
valuable: it provides a systematic mechanism for applying 
the ISS analysis to objectives that are not originally 
strongly convex, with explicit control over the trade-off 
between the regularisation bias and the 
convergence neighbourhood radius through the parameters 
$c$, $h$, $\mu$, and $t$. Extending the analysis to 
parameter regimes where the regularised objective is not 
strongly convex is an interesting direction for future work.
\end{remark}

\begin{remark}\label{rem:reg_HBNAG}
    The same approach of Section~\ref{sec:L2reg} applies to the regularised problem~\eqref{eq:reg_main} with HB or NAG: one can show that their ZO counterparts are perturbed versions of the corresponding FO algorithms with bounded and controllable perturbations, yielding results analogous to Theorem~\ref{th:GD_reg}. The details are omitted for brevity.
\end{remark}

\section{Numerical examples complementary details}\label{app:exps}
In this section, we present complementary experiments and results for the numerical examples given in Section~\ref{sec:exps}. All the tests have been run on a Dell Latitude 7430 Laptop with a 12th Gen Intel Core i7 CPU.

\subsection{Additional experiments for the quadratic objective}\label{app:quad_exp}
This section complements the quadratic experiment in Section~\ref{sec:quad_exp} with additional parameter studies.

\paragraph{Effect of the smoothing parameter $\mu$.} We fix $t=1$ and $h=10^{-5}$ and vary $\mu \in \{10^{-1},4\times10^{-3},10^{-6},10^{-8}\}$. The results are shown in Figure~\ref{fig:param_mu}. For large $\mu$, ZO-GD fails to converge; reducing $\mu$ restores convergence. However, once $\mu$ is sufficiently small, further reduction has negligible effect, since the dominant term in the perturbation bound no longer depends on $\mu$. This is in contrast to $h$, which controls both terms in the perturbation bound.

\begin{figure}[htb]
	\centering \includegraphics[width=0.9\textwidth]{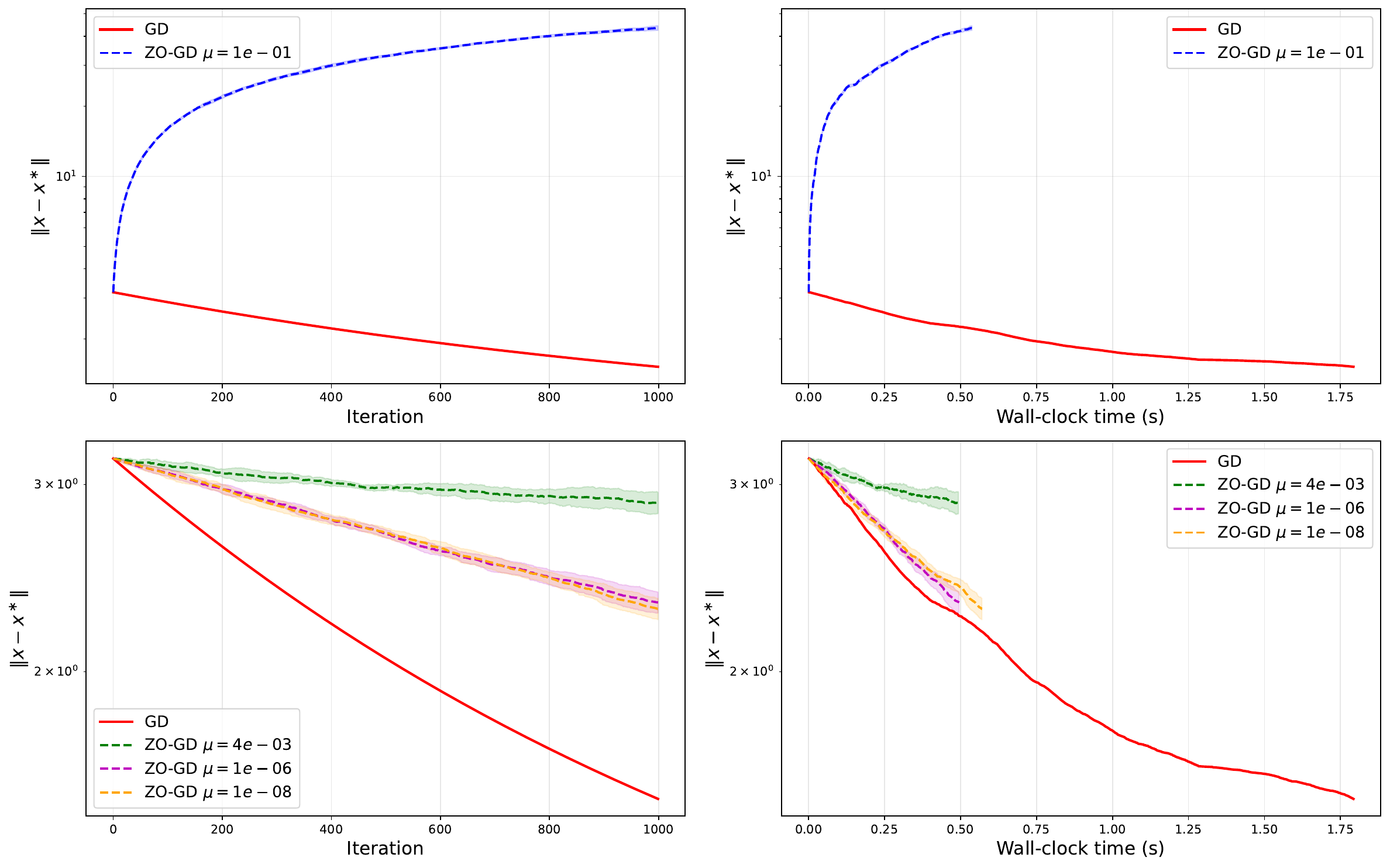}
	\caption{Parameter analysis: smoothing parameter.}\label{fig:param_mu}
\end{figure}

\paragraph{Effect of the number of sampled directions $t$.} We fix $\mu=h=10^{-5}$ and vary $t \in \{1,5,10,20\}$. The results are shown in Figure~\ref{fig:param_t}. Increasing $t$ reduces the perturbation, and ZO-GD converges to a neighbourhood of GD that shrinks with $t$.

\begin{figure}[htb]
	\centering \includegraphics[width=0.9\textwidth]{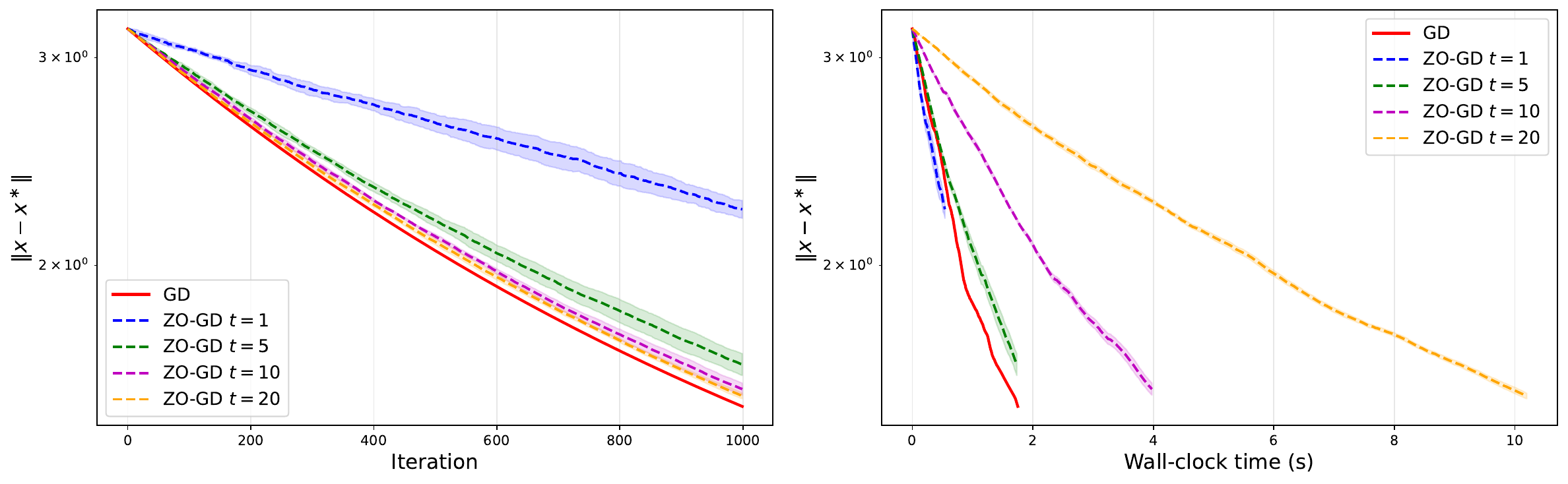}
	\caption{Parameter analysis: number of sampled directions.}\label{fig:param_t}
\end{figure}

\paragraph{Parameter compensation.} We revisit the cases in which GD converged, but ZO-GD diverged, demonstrating that the perturbation can always be controlled by adjusting the remaining parameters. First, consider the divergent case $h=10^{-4}$, $\mu=10^{-5}$, $t=1$ from Figure~\ref{fig:param_lr}. As shown in the first row of Figure~\ref{fig:param_int}, increasing the number of sampled directions to $t=25$ and reducing $\mu=10^{-7}$ brings the perturbation within the admissible range. Second, consider the divergent case $\mu=0.01$, $t=1$, $h=10^{-5}$ from Figure~\ref{fig:param_mu}. Since $\mu$ and $t$ do not interact directly in the perturbation bound, we instead reduce $h$ to $10^{-7}$. The result, shown in the second row of Figure~\ref{fig:param_int}, confirms that ZO-GD again tracks GD.

\begin{figure}[htb]
	\centering \includegraphics[width=0.8\textwidth]{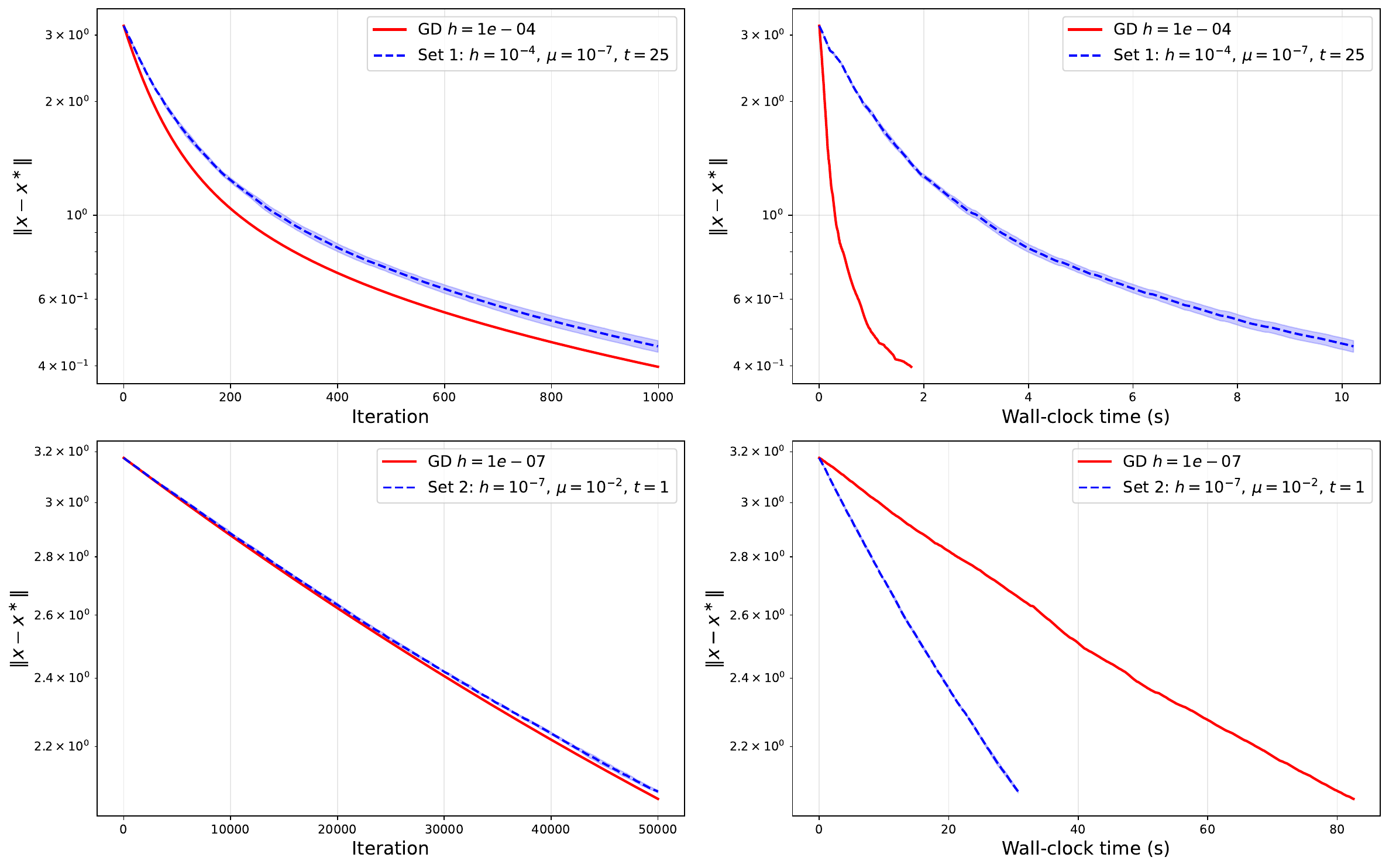}
	\caption{Parameter analysis: compensation for divergent cases.}\label{fig:param_int}
\end{figure}



\subsection{Additional experiments for neural network classification}\label{app:nn_exp}
This section complements the neural network experiment in Section~\ref{sec:nn_exp} by providing a parameter sensitivity analysis.

\paragraph{Parameter sensitivity.} We conduct parameter sweeps for ZO-GD analogous to those in Appendix~\ref{app:quad_exp}. Figure~\ref{fig:nn_param} shows the effect of varying $h$, $\mu$, and $t$ individually while holding the other two fixed, for the network with $100,609$ parameters. The same phenomena observed for the quadratic objective are reproduced: when $h$ is too large, ZO-GD diverges, and reducing $h$ brings ZO-GD progressively closer to GD (Figure~\ref{fig:nn_param}, left). Reducing $\mu$ restores convergence when it is initially too large, but further reduction has diminishing effect once the variance term dominates the perturbation bound (Figure~\ref{fig:nn_param}, middle). Increasing the number of sampled directions $t$ shrinks the neighbourhood, with ZO-GD converging closer to GD for larger $t$ (Figure~\ref{fig:nn_param}, right). These observations are consistent with the two-term structure of the perturbation bound in Lemma~\ref{lem:q_GDreg}: $h$ and $\mu$ control the smoothing bias term, while $h$ and $t$ control the variance term.

\begin{figure}[htb]
	\centering \includegraphics[width=1\textwidth]{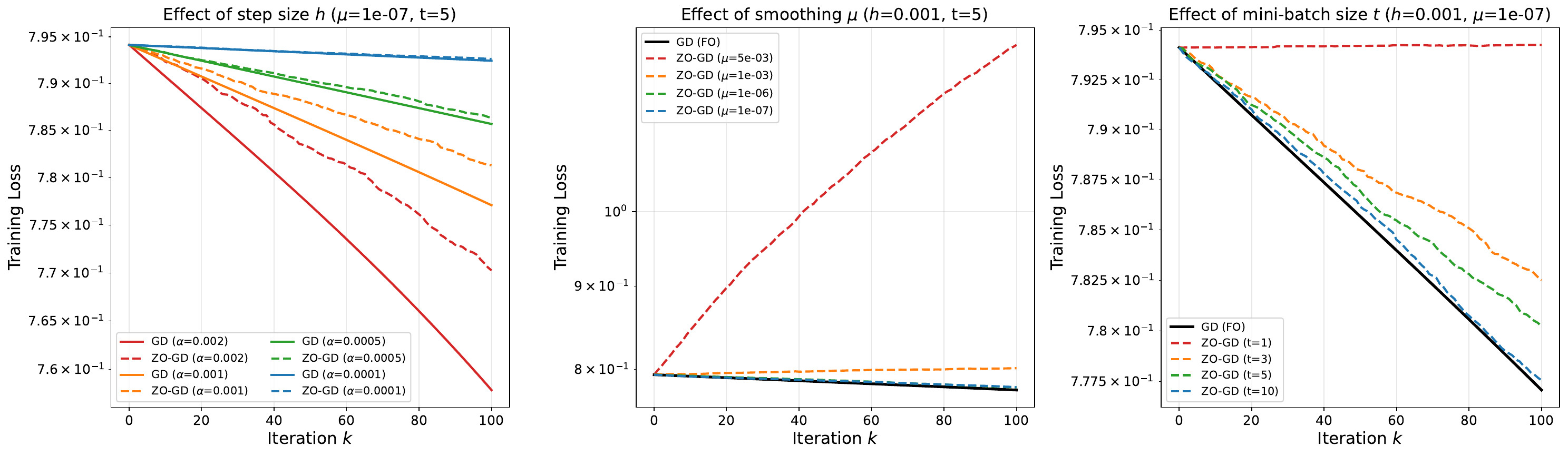}
	\caption{Parameter sensitivity for ZO-GD on the neural network. Left: effect of $h$. Centre: effect of $\mu$. Right: effect of $t$. Solid black: GD (FO); dashed coloured: ZO-GD.}\label{fig:nn_param}
\end{figure}

\newpage

\end{document}